\documentclass{amsart}[12pt.]

\usepackage[linktoc=none]{hyperref}

\usepackage{amsmath}
\usepackage{amssymb}
\usepackage{mathrsfs}
\usepackage[shortalphabetic]{amsrefs}
\usepackage[dvipsnames]{xcolor}

\newtheorem{thm}{Theorem}[section]
\newtheorem{cor}[thm]{Corollary}
\newtheorem{conj}[thm]{Conjecture}
\newtheorem{lem}[thm]{Lemma}
\newtheorem{prop}[thm]{Proposition}

\newtheorem{rmk}[thm]{Remark}

\makeatletter
\@namedef{subjclassname@2020}{\textup{2020} Mathematics Subject Classification}
\makeatother

\makeatletter
\newcommand\thankssymb[1]{\textsuperscript{\@fnsymbol{#1}}}
\makeatother

\begin{document}

\title{Infinitude of palindromic almost-prime numbers}

\author[Aleksandr Tuxanidy]{Aleksandr Tuxanidy \thankssymb{2}}
\thanks{\thankssymb{2} Corresponding author}

\author{Daniel Panario}

\iffalse
\fi

\address{School of Mathematics and Statistics, Carleton University, 1125 Colonel By Drive, Ottawa, Ontario, K1S 5B6, Canada}
\email{AleksandrTuxanidyTor@cmail.carleton.ca, daniel@math.carleton.ca}

	\begin{abstract}
	It is proven that, in any given base, there are infinitely many palindromic numbers having at most six prime divisors, each relatively large. The work involves equidistribution estimates for the palindromes in residue classes to large moduli, offering upper bounds for moments and averages of certain products closely related to exponential sums over palindromes. 
\end{abstract}

	\maketitle
	
	\setcounter{tocdepth}{1}

\tableofcontents

%%%%%%%%%%%%%%%%%%%%%%%%%%%%%%%%%%%%%%%%%%%%%%%%%%%%%%%%%%%%%%%%%%%%%%%%%%%%%%%%%%%%%%%%%%%%%%%%%%%%%%%%%%%%%%%%%%%%%%%%%%%%%%%%%%%%%%%%%%%%%%%%%%%%%%%%%%%%%%%%%%%%%%%%%%%%%%%%%%%%%%%%%%%%%%%%%%%%

\section{Introduction}

\subsection{Tattarrattat}

Generally speaking, a {\em palindrome} is a finite sequence of objects which reads the same forwards as backwards. Any such sequence is said to be {\em palindromic}. For example the sequence of the (decimal) digits of the prime number
$$9609457639843489367549069$$ 
is a palindrome. According to SigmaGeek \cite{SigmaGeek} (see also Iwao \cite{Google}) it is the largest known palindromic prime appearing in the decimal expansion of the digits of $\pi$. 

One of the earliest historical palindromic objects, a SATOR square, was found in the remains of the ancient city of Pompeii (see for instance Sheldon \cite{Sheldon}) destroyed in the year 79 AD after the eruption of Mount Vesuvius. Palindromes appear in several facets of human endeavor, for instance in music (e.g. Joseph Haydn's Symphony No. 47, ``The Palindrome") in literature (e.g. Georges Perec's ``Le Grand Palindrome''). In recent decades, palindromes are studied intensively in various areas of research \cite{Banks, Banks-Hart-Sakata, Banks-Shparlinski, Cilleruelo, Col, Fici, Gao, Garefalakis, Gusfield, Irving, Porto, Rajasekaran}.
They also appear in nature; for example, unusually large portions of the human X and Y chromosomes are palindromic (in a slightly different sense; see Larionov-Loskutov-Ryadchenko \cite{Larionov}). 

\subsection{Conjecture and previous work} 

From here on out, $b$ denotes an integer larger than $1$. One says that a natural number $n$ is {\em $b$-palindromic}, or is a {\em $b$-palindrome}, if the sequence of its $b$-adic digits is palindromic. When the context is clear we simply say that $n$ is palindromic (or a palindrome).

Palindromic numbers belong, more broadly, in the category of sequences and functions defined in terms of digit constraints in integer bases. 
These have been studied abundantly in several studies. We mention for example the work of Bourgain \cite{Bourgain} and Swaenepoel \cite{Swaenepoel} on primes with prescribed digits, the recent work of Dartyge et al \cite{Dartyge} on reversible primes, Mauduit-Rivat \cite{maud} on the sum of digits of primes, and Maynard \cite{Maynard} on primes with missing digits.
We refer the reader to Dartyge et al \cite{Dartyge} for several other interesting works and references therein. With regards to palindromic numbers, we mention for instance the work of Banks-Hart-Sakata \cite{Banks-Hart-Sakata} showing that almost all palindromes are composite (see (\ref{eqn: BHS density}) below) Banks-Shparlinski \cite{Banks-Shparlinski} on palindromes with many prime factors, Cilleruelo-Luca-Baxter \cite{Cilleruelo} showing that if $b \geq 5$, every positive integer can be written as a sum of at most three $b$-palindromes, Col \cite{Col} on the distribution of palindromes in residue classes and palindromic almost-primes (see Theorems \ref{thm: Col}, \ref{thm: Col equi estimate} and (\ref{eqn: Col density}) below) and Irving \cite{Irving} giving a lower bound for the number of $b$-palindromic semiprimes in the interval $(b^2,b^3)$ with $b$ sufficiently large.

On the topic of palindromic numbers, perhaps one of the most difficult unsolved problems is the following conjecture (see for instance the concluding remarks in Banks-Hart-Sakata \cite{Banks-Hart-Sakata} and Banks-Shparlinski \cite{Banks-Shparlinski}). First for a real number $x \geq 1$, we let $\mathscr{P}_b(x)$ be the set of all positive $b$-palindromic integers at most $x$ in size.

\begin{conj}\label{conj: pal primes}
	Let $b \geq 2$ be an integer. Then there are infinitely many prime numbers that are $b$-palindromic. In fact
	$$
	\#\left\{p \in \mathscr{P}_b(x) \right\} \asymp_b  \dfrac{\#\mathscr{P}_b(x)}{\log x}
	$$ 
	for $x$ large.
\end{conj}

Conjecture \ref{conj: pal primes} seems far out of reach with current methods. As pointed out by Banks-Shparlinski \cite{Banks-Shparlinski}, one of the main difficulties stems from the large sparsity of $b$-palindromic numbers, hindering the efficacy of analytic methods. The sparsity of these is, in essence, as large as that of numbers with form $n^2+1$. However Landau's problem (posed in the year 1912) of whether or not there are infinitely many such primes, remains unsolved. On the other hand, Iwaniec \cite{Iwaniec} proved there are infinitely many integers $n$ such that $n^2+1$ has at most two primes factors. In view of recent results (see the discussion below and the following subsection) we believe that proving the analogous result, for $b$-palindromes, may be feasible over the upcoming years.

Some progress was made in the direction of Conjecture \ref{conj: pal primes}. In the year 2004, Banks-Hart-Sakata \cite{Banks-Hart-Sakata} showed that for $x$ large enough, 
\begin{equation}\label{eqn: BHS density}
\#\left\{p \in \mathscr{P}_b(x)\right\}\ll_b \dfrac{\log \log \log x}{\log \log x}\#\mathscr{P}_b(x).
\end{equation}
In 2009, Col \cite{Col} improved this to the expected
\begin{equation}\label{eqn: Col density}
\#\left\{p \in \mathscr{P}_b(x)\right\} \ll_b \dfrac{\#\mathscr{P}_b(x)}{\log x}.
\end{equation}
Col also obtained (see Theorem \ref{thm: Col} below) a lower bound for the amount of $b$-palindromic almost-primes, when the number of prime divisors is bounded above by a certain value depending only on $b$. In the context of the infinitude of such palindromes, Theorem \ref{thm: Col} was the first, and up to now unique, result of its kind. Here $\Omega(n)$ denotes the number of prime divisors of $n$, counted with multiplicity. 

\begin{thm}[Col \citelist{\cite{Col},  Corollary 2}]\label{thm: Col} 
	We have
	$$
	\#\left\{n \in \mathscr{P}_b(x) \ : \ \Omega(n) \leq k_b \right\} \gg_b \dfrac{\#\mathscr{P}_b(x)}{\log x}
	$$
	for some value $k_b > 0$ depending only on $b$ and satisfying $k_b \sim 24\pi b$ as $b \to \infty$. In particular $k_2=60$ and $k_{10} = 372$.
\end{thm}

Theorem \ref{thm: Col} was a consequence of Col's equidistribution estimate (see Theorem \ref{thm: Col equi estimate} below) for $b$-palindromes in residue classes to large moduli, combined with well-established facts from sieve theory. Here we use the notation
$$\mathscr{P}_b(x,a,q) := \left\{n \in \mathscr{P}_b(x) \ : \ n \equiv a (q)\right\}.$$ 

\begin{thm}[Col \citelist{\cite{Col},  Theorem 2}]\label{thm: Col equi estimate}
	There exists some value $\beta > 0$, depending only on $b$, such that for any $\epsilon,A > 0$,
	$$
	\sum_{\substack{q \leq x^{\beta - \epsilon} \\ (q,b^3-b)=1}} \sup_{y \leq x} \max_{a \in \mathbb{Z}} \left|\#\mathscr{P}_b(y,a,q) - \dfrac{\#\mathscr{P}_b(y)}{q}\right| \ll_{\epsilon,A,b} \dfrac{\#\mathscr{P}_b(x)}{\log^A x}. 
	$$
	For $b$ large, we may take $\beta \sim 1/6\pi b $. 
	Setting $\beta_b = \beta - \epsilon$, the inequality above holds for the following particular values.
	\begin{center} 
	\begin{tabular}{|c|c|c|c|c|c|c|c|c|c|}
	\hline	
	$b$ & $2$ & $3$ & $4$ & $5$ & $6$ & $7$ & $8$ & $9$ & $10$\\
	\hline
	$\beta_b$ &  $1/30$& $1/94$ &  $1/74$& $1/122$ & $1/114$ &  $1/158$&  $1/150$&  $1/194$&$1/186$\\
	\hline
	\end{tabular}
	\end{center}
\end{thm}	

Theorem \ref{thm: Col equi estimate} (see also Proposition \ref{prop: Linfty bound} here) considerably improved over previous estimates of Banks-Hart-Sakata \cite{Banks-Hart-Sakata} valid for $\beta \ll (\log \log x)/\log x$ with the implied constant small enough (see Corollary 4.5 in \cite{Banks-Hart-Sakata}).

\subsection{Statement of results}

In the present work we obtain the following improved versions of Theorems \ref{thm: Col} and \ref{thm: Col equi estimate}. Here $P^-(n)$ denotes the smallest prime divisor of $n$.

\begin{thm}\label{thm: at most 6 primes}
	Let $b\geq 2$ be an integer. Then for $x \gg_b 1$ large enough, 
	$$
	\#\left\{n \in \mathscr{P}_b(x) \ : \ \Omega(n) \leq 6, \ P^{-}(n) \geq x^{1/21}\right\} \asymp_{b} \dfrac{\#\mathscr{P}_b(x)}{\log x}.
	$$
\end{thm} 

Theorem \ref{thm: at most 6 primes} is a consequence of Theorem \ref{thm: equidistribution} below (discussed in the following subsections) facts from sieve theory and an analogue of Theorem \ref{thm: equidistribution} with square moduli. See Proposition \ref{prop: square equidistribution} for this. The latter is an application of the Baier-Zhao \cite{Baier-Zhao}
estimate for the large sieve with square moduli, as well as the $L^\infty$-type bound of Col \cite{Col} (see Proposition \ref{prop: Linfty bound} here).

In what follows
$$\mathscr{P}_b^*(x) := \{n \in \mathscr{P}_b(x) \ : \ (n,b^3-b)=1\}.$$

\begin{thm}\label{thm: equidistribution}
	Let $b\geq 2$ be an integer and let $x \geq 1$. Then for any $\epsilon > 0$,
	\begin{equation}\label{eqn: equi 1}
	\sum_{\substack{q \leq x^{1/5-\epsilon} \\ (q,b^3-b)=1}} \sup_{y \leq x }\max_{a \in \mathbb{Z}} \left|\sum_{\substack{n\in \mathscr{P}_b^*(y)}}\left(\mathbf{1}_{n \equiv a (q)} - \dfrac{1}{q}\right)\right| \ll_{b,\epsilon} \dfrac{\#\mathscr{P}_b^*(x)}{e^{\sigma(b, \epsilon) \sqrt{\log x}}} ,
	\end{equation}
	where $  \sigma(b,\epsilon) > 0$ is some value depending only on $b$ and $\epsilon$. 
\end{thm}

We now discuss some of the main
ingredients in the proof of Theorem \ref{thm: equidistribution}.

\subsection{Setup}

For the sake of brevity we denote by $E$ the left hand side of the inequality in Theorem \ref{thm: equidistribution}. Let $\mathscr{P}_b^0$, $\mathscr{P}_b^*$ be defined as in (\ref{def of P0}), (\ref{def of P*}), respectively. Set $\mathscr{P}_b^0(x) := \mathscr{P}_b^0 \cap [1,x]$. Since every $b$-palindrome $n \geq 1$ with $\lfloor \log_b n \rfloor$ odd (or equivalently, with an even number of digits in its $b$-adic expansion) is divisible by $b + 1$, a divisor of $b^3-b$, it follows
$
\mathscr{P}_b^* = \left\{n \in \mathscr{P}_b^0 \ : \ (n,b^3-b)=1\right\}.  
$
Then by
the M\"{o}bius inversion formula $$\mathbf{1}_{(n,b^3-b)=1} = \sum_{r \mid (n, b^3-b)} \mu(r),$$ 
the Fourier expansions 
$$
\mathbf{1}_{r \mid n} = \dfrac{1}{r}\sum_{k(r)}e_r(nk)
$$ and 
\begin{equation}\label{eqn: Fourier type expansion}
\mathbf{1}_{n\equiv a (q)} = \dfrac{1}{q}\sum_{h(q)} e_q(-ah)e_q(n h)
\end{equation}
(with the common notation $e_m(\alpha) := e(\alpha/m)$ with $e(\alpha) := e^{2\pi i \alpha}$)
one derives
$$
E \ll_b(\log x) \max_{k\in \mathbb{Z}}\sum_{\substack{2 \leq q \leq x^{1/5 - \epsilon} \\ (q,b^3-b)=1}} \dfrac{1}{q} \sum_{h(q)}^*\sup_{\substack{y \leq x }} \left|\sum_{n \in \mathscr{P}_b^0(y)} e\left(\dfrac{hn}{q} + \dfrac{kn}{b^3-b} \right)\right| .
$$
The asterisk above the $h$-sum signifies that $h$ runs over the invertible residue classes modulo $q$. 

The use of the decomposition in (\ref{eqn: Fourier type expansion}) is rather typical in the literature and it will shortly lead to qualitative gains in multiplicative algebraic structures, with the appearance of the products $\Phi_N$ defined in (\ref{eqn: def of Phi}). Unfortunately we pay the price of a quantitative loss of roughly $x^{1/5-\epsilon}$ in the trivial bound. Thus in order for such a Fourier-type approach to succeed, we must not only be able to recover the lost $x^{1/5-\epsilon}$, but gain much more (in fact gain an extra $\exp(\sigma\sqrt{\log x})$). 

From the works of Banks-Hart-Sakata \cite{Banks-Hart-Sakata} and Col \cite{Col} (see Lemma \ref{lem: bound for exp sum of pals} here) the inner sum admits a decomposition into linear combinations of the products $\Phi_N$ with $0 \leq N \leq \frac{1}{2}\log_b x$. After a dyadic subdivision of the interval $2 \leq q \leq x^{1/5 - \epsilon}$, the problem of bounding $E$ then boils down to bounding averages of the form
$$
S = \sum_{\substack{2 \leq q \leq Q \\ (q,b^3-b)=1}} \sum_{h(q)}^* \Phi_N\left(\dfrac{h}{q} + \dfrac{k}{b^3-b}\right) 
$$ 
for $Q \ll x^{1/5-\epsilon}$. The trivial bound gives $S \leq Q^2 b^N \leq Q^2 \sqrt{x}$ and our job is to beat this by much more than a factor of $Q$.

Thanks to the work of Col \cite{Col} (see Proposition \ref{prop: Linfty bound} here) we can readily dispose of the cases with $Q$ relatively small, say $Q \ll \exp(c\sqrt{\log x})$. We are thus left with considering $S$, or more strongly the average
\begin{equation}\label{eqn: def of T}
T = \sum_{\substack{q \leq Q \\ (q,b)=1}} \sum_{h(q)}^* \Phi_N\left(\dfrac{h}{q} + \beta \right),
\end{equation}
for $\exp(c\sqrt{\log x}) \ll Q \ll x^{1/5-\epsilon}$ and $\beta \in \mathbb{R}$. 

As a first rough thought, one could envision splitting the product $\Phi_N$ into two products, one over $1 \leq n \leq M$ and the other over $M < n < N$, applying Cauchy-Schwarz and end up considering averages of squares (just as Bombieri-Vinogradov's treatment of the Type II sums; see Iwaniec-Kowalski \cite{Iwaniec-Kowalski}). These could be attacked say via a direct argument (see for instance Lemma \ref{lem: L2 bound}) or the large sieve inequality (see Lemma \ref{lem: large sieve} below). The direct argument (namely that of expanding the square, switching orders of summation, using the orthogonality of the additive characters and bounding the resulting sums trivially) can at best aid in recovering a $Q$ from the trivial bound, and is thus insufficient. This brings us to the large sieve inequality.

\subsection{Large sieve}

First given a real number $0 < \delta \leq 1/2$, one says that a set of points $\alpha_1, \ldots, \alpha_R \in \mathbb{R}/\mathbb{Z}$ is $\delta${\em -spaced} if $\|\alpha_r - \alpha_s\| \geq \delta$ whenever $r \neq s$, where $\|\alpha\| := \min_{k\in \mathbb{Z}}|\alpha - k|$. 

\begin{lem}[Large sieve inequality]\label{lem: large sieve}
	For any set of $\delta$-spaced points $\alpha_1, \ldots, \alpha_R \in \mathbb{R}/\mathbb{Z}$ and any complex numbers $a_n$ with $M < n \leq M + N$, where $0 < \delta \leq 1/2$ and $N \geq 1$ is an integer, we have
	\begin{equation}\label{ineq: large sieve}
	\sum_{r=1}^R \left|\sum_{ M < n \leq M+N} a_n e(\alpha_r n)\right|^2 \leq \left(\delta^{-1} + N - 1\right) \sum_{M< n \leq M+N}|a_n|^2. 
	\end{equation}
\end{lem}

\begin{proof}
	See for instance Theorem 7.7 in Iwaniec-Kowalski \cite{Iwaniec-Kowalski}. 	
\end{proof}

The large sieve inequality is a powerful and useful tool for tackling averages of exponential sums and has enjoyed several applications in the literature (see for instance \cite{Friedlander-Iwaniec, Iwaniec-Kowalski, maud, Maynard, Tenenbaum} to name a few). Unfortunately the quality of the bounds it offers highly depends on the density $\eta = \#\mathcal{S}/N$ of the set $\mathcal{S} \subseteq (M, M+N]$ in which the sequence $(a_n)$ is supported on, with better bounds for denser $\mathcal{S}$ and worse bounds for sparser $\mathcal{S}$. For example in the case when each $a_n \ll 1$ and $R \asymp Q^2 \asymp \delta^{-1}  \asymp N$, the large sieve inequality beats the trivial bound by a factor of $\#\mathcal{S} = \eta N \asymp Q \eta \sqrt{N}$. In particular in order for the large sieve to beat the trivial bound by much more than $Q$ (which is what we seek) one needs $\eta$ to be much larger than $1/\sqrt{N}$. Unfortunately sparse sets such as, say, the $b$-palindromes or close relatives, have density $ \ll 1/\sqrt{N}$ on $(M, M+N]$. For sequences supported on such sets, direct use of the large sieve proves insufficient.

An important idea for the case of the sparse sequences is to consider instead larger moments such as
$$
\sum_{r=1}^R \left|\sum_{ M < n \leq M+N} a_n e(\alpha_r n)\right|^{2K}
$$
for some integer $K \geq 2$; see for example Iwaniec-Kowalski \cite{Iwaniec-Kowalski}. One may arrive at these from smaller moments via H\"{o}lder's inequality, say. To explain the concept, note we may write 
$$
\left(\sum_{ M < n \leq M+N} a_n e(\alpha_r n)\right)^K = \sum_{KM < n \leq K(M+N)} b_n e(\alpha_r n),
$$
where the coefficients $b_n$ are now supported on the set $K\mathcal{S}$ comprised of the additions of any $K$ elements from $\mathcal{S}$. If the set $\mathcal{S}$ is sufficiently dissociated additively, then $K\mathcal{S}$ is much denser in $(KM, K(M+N)]$. In this case we can make a more effective use of the large sieve.

\subsection{Approach}

Let us now return to (\ref{eqn: def of T}). 
Motivated by the discussion above, we split the product $\Phi_N$ into three pieces as $\Phi_N = P_1P_2P_3$, where $P_1$ is the product over $1 \leq n \leq M$, $P_2$ is the product over $M < n \leq 2M$ and $P_3$ is the product over $2M < n < N$ (actually this is our approach for the case when $Q \gg b^{(\frac{2}{5} - \epsilon_0)N}$; the case $Q \ll b^{(\frac{2}{5} - \epsilon_0)N}$ is treated in a slightly different way). The choice of $M$ is optimized later. Applying H\"{o}lder's inequality with the triple $(\ell,\ell,2K)$ (where $K\geq 2$ is an integer to be chosen later and $\ell > 2$ is such that $2/\ell + 1/2K = 1$) one obtains
$$
T \leq \left(\sum_{\substack{q \leq Q \\ (q,b)=1}} \sum_{h(q)}^* P_1^\ell \right)^{1/\ell} \left(\sum_{\substack{q \leq Q \\ (q,b)=1}} \sum_{h(q)}^* P_2^\ell \right)^{1/\ell} \left(\sum_{\substack{q \leq Q \\ (q,b)=1} } \sum_{h(q)}^* P_3^{2K}\right)^{1/2K}. 
$$
For $K$ somewhat large (although bounded in terms of some parameters for our purposes) $\ell$ is close to $2$. The first two sums corresponding to $P_1, P_2$ are then treated by us in direct fashion (as in Lemma \ref{lem: L2 bound}). Here we avoid using the large sieve which, given the sparsity of the coefficient sequences corresponding to $P_1, P_2$, gives no significant gains over the direct approach.

There are two reasons why we have chosen to raise $P_3$ to the $2K$-th power and not $P_1$ nor $P_2$, namely: 1. the following algebraic identity
$$
\sum_{\substack{q \leq Q \\ (q,b)=1}} \sum_{h(q)}^*P_3^{2K} = \sum_{\substack{q \leq Q \\ (q,b)=1}} \sum_{h(q)}^* \Phi_{N-2M}^{2K}\left(\dfrac{h}{q} + \beta b^{2M}\right); 
$$
and 2. the sequence of coefficients corresponding to $\Phi_{N-2M}$ is denser than those corresponding to $P_1,P_2$. 
The trivial bound shows the above is at most $ Q^2 b^{2K(N-2M)} \leq Q^2 x^{K(1 - \frac{2M}{N})}$ and our main goal then reduces to beating this by much more than $Q$ for certain suitable choices of $M,K$. This is accomplished via Proposition \ref{prop:2K-th moment2} giving upper bounds for moments of $\Phi_N$ for any $N\geq 2$. Its proof relies partly on the large sieve.

\section{Acknowledgments}

We would like to express our gratitude to Qiang Wang for useful advice and to Igor E. Shparlinski for conversations involving the square moduli and notifying us of Col's work \cite{Col}. We are also deeply indebted to C\'{e}cile Dartyge, Hirotaka Kobayashi, Yuta Suzuki and Ryota Umezawa for pointing out one important omission in an earlier draft of the work and for referring us to the books by Greaves \cite{Greaves} and Halberstam-Richert \cite{Halberstam-Richert}. Finally we thank the anonymous referees for several helpful suggestions.  
D. Panario was funded by the Natural Science and Engineering Research Council of Canada (NSERC), reference number RPGIN-2018-05328.

\section{Paper structure}

The rest of the work goes as follows. 
In Section \ref{sec: notation} we describe the notation used in the work. Most of it is typical in the literature, but we also introduce notation that is particular to us. 
In Section \ref{sec: preparatory lemmas} we compile several technical tools needed in the following two sections. In Section \ref{sec: exp sums} we discuss the current (but limited) state of knowledge on exponential sums over palindromes, as well as slightly add to the literature via Proposition \ref{prop: product bound for modulus p} and Corollary \ref{cor: bound using product of Bourgain sequence} there. In Section \ref{sec: moments} we prove Proposition \ref{prop:2K-th moment2} giving an upper bound for moments of $\Phi_N$. In Section \ref{sec: average bound} we prove Proposition \ref{prop: average bound} bounding an average of $\Phi_N$. In Section \ref{sec: equi estimate} we prove Theorem \ref{thm: equidistribution} giving average equidistribution estimates for palindromes in residue classes. In Section \ref{sec: square} we prove Proposition \ref{prop: square equidistribution} giving equidistribution estimates for square moduli. We then conclude the work in Section \ref{sec: final proof} with a proof of Theorem \ref{thm: at most 6 primes}.

\section{Notation}\label{sec: notation}

Here we describe the notations used in the work, some of which we already encountered; these are restated here for the sake of clarity.

\subsection{Literary notation}

We use the asymptotic notations $X \ll Y$, $Y \gg X$, $X = O(Y)$, all signifying here that $|X| \leq C|Y|$ for some unspecified constant $C > 0$. The notation $X \asymp Y$ is used to assert that both $X \ll Y$ and $Y \ll X$, simultaneously. Dependence of the implied constants on parameters will be denoted by a subscript. The symbol $o(X)$ denotes a quantity satisfying $o(X)/X \to 0$ as $X \to \infty$.  

The letters $\mu, \tau, \varphi$ denote the M\"{o}bius, divisor and Euler's totient functions, respectively. For 
an integer $n \geq 1$, the symbol $\Omega(n)$ denotes the total number of prime factors of $n$, counting their multiplicity. The symbol $P^-(n)$ denotes the smallest prime factor of $n$. The letter $p$ is reserved to denote a prime number.

The notation $n \equiv a (q)$ denotes the assertion that $n$ is congruent to $a$ modulo $q$. We use $d \mid n$ to assert that $d$ divides $n$, and $(m,n)$ to represent the greatest common divisor (GCD) of $m$ and $n$. If $(m,n)=1$, we use the notation $\text{ord}_m(n)$ for the multiplicative order of $n$ modulo $m$.  

The symbol 
$$\sum _{a (q)}$$
denotes a summation over all the residue classes $a$ modulo $q$, whereas 
$$\sum_{a(q)}^*$$ 
denotes a sum over the invertible residue classes $a$ modulo $q$. 

Given a function $f : \mathbb{R} \to \mathbb{C}$, we denote by $\|f\|_1$ its $L^1$-norm and by $\|f\|_\infty$ its $L^{\infty}$-norm (whenever these exist).

For a real number $x$ and an integer $q \geq 1$, we use the notation $e(x) := e^{2\pi i x}$ for the complex exponential, $e_q(x) := e(x/q)$ for the additive character modulo $q$ and $\exp(x) := e^x$ for the exponential. If $x > 0$, its natural logarithm is denoted by $\log x$ and its base-$b$ logarithm by $\log_b x$.

We use $\|x\| $ to symbolize the distance from $x$ to its nearest integer; that is, $\|x\| := \min_{k \in \mathbb{Z}}|x-k|$. Moreover $\lfloor x \rfloor$ denotes the floor of $x$ and $\{x\} := x - \lfloor x \rfloor$ its fractional part. 

We use the standard trigonometric notations $\csc(x) = 1/\sin(x)$ and $\cot(x) = \cos(x)\csc(x)$.

Given a statement $E$, we use $\mathbf{1}_E$ to symbolize the indicator function of $E$; that is, $\mathbf{1}_E = 1$ if $E$ is true and $\mathbf{1}_E = 0$ if $E$ is false. For example, $\mathbf{1}_{n \equiv a (q)}$ equals $1$ if $n \equiv a (q)$ and is zero otherwise. 

Given real numbers $\alpha_1, \ldots, \alpha_N$, their {\em discrepancy} $D_N(\alpha_1, \ldots, \alpha_N)$ is defined by
\begin{equation}\label{eqn: discrepancy}
	D_N(\alpha_1, \ldots, \alpha_N) := \sup_{0 \leq c \leq d \leq 1}\left|\dfrac{ \# \left\{1 \leq n \leq N \ : \ c \leq \{\alpha_n\} < d  \right\}}{N} - (d-c)\right|
\end{equation}
and their {\em star-discrepancy} $D_N^*(\alpha_1, \ldots, \alpha_N)$ is defined as
\begin{equation}\label{eqn: star discrepancy}
	D_N^*(\alpha_1, \ldots, \alpha_N) := \sup_{0 \leq d \leq 1} \left|\dfrac{ \# \left\{1 \leq n \leq N \ : \ \{\alpha_n\} < d \right\}}{N} - d\right|.
\end{equation}

\subsection{Non-standard notation}

The letter $b$ is reserved to denote an integer larger than $1$. The notation $\mathscr{P}_b$ symbolizes the set of all positive $b$-palindromic integers. We define the sets $\mathscr{P}_b^0, \mathscr{P}_b^*$, as follows: 
\begin{align}
	\mathscr{P}_b^0 &:= \left\{ n \in \mathscr{P}_b  \ : \ \lfloor \log_b(n) \rfloor \equiv 0 (2) \right\},
	\label{def of P0}
	\\
	\mathscr{P}_b^* &:= \left\{ n \in \mathscr{P}_b \ : \ (n,b^3-b)=1 \right\}. \label{def of P*}
\end{align}
From the definition above one observes that $\mathscr{P}_b^0$ is comprised exactly of all positive $b$-palindromes with an odd number of digits (in their $b$-adic expansion). 
Since every $b$-palindrome with an even number of digits is divisible by $b + 1$, and $b + 1$ is a divisor of $b^3-b$, it follows $\mathscr{P}_b^* \subset \mathscr{P}_b^0$. 
In fact
\begin{align*}
	\mathscr{P}_b^0 &= \bigcup_{N=0}^\infty \Pi_b(2N),\\
	\mathscr{P}_b^* &= \bigcup_{N=0}^\infty \Pi_b^*(2N),
\end{align*}
where for an integer $N \geq 0$, 
\begin{align}
	\Pi_b(N) &:= \mathscr{P}_b \cap \left[b^N, b^{N+1}\right), \label{eqn: def of pi} \\
	\Pi_b^*(N) &:= \left\{ n \in \Pi_b(N) \ : \ (n,b^3-b)=1  \right\}. \label{eqn: def of pi star}
\end{align}

Given a set $\mathscr{A}$ of integers, a real number $x$ and integers $a,q$, we define the sets
\begin{align*}
	\mathscr{A}(x) &:= \mathscr{A} \cap \left\{ n \in \mathbb{Z} \ : \ n \leq x\right\},\\
	\mathscr{A}(x,a,q) &:= \left\{n \in \mathscr{A}(x) \ : \ n \equiv a (q)  \right\}.
\end{align*} 
Thus for instance 
\begin{align}
	\mathscr{P}_b^*(x) &:= \mathscr{P}_b^* \cap [1,x]\label{def: P star},\\
	\mathscr{P}_b^*(x,a,q) &:= \left\{ n \in \mathscr{P}_b^*(x) \ :\ n \equiv a (q)   \right\},\label{def: P star cong}
\end{align}
with our notation. 

Given a real number $\alpha$ and an integer $N\geq 0$, we define
\begin{equation}\label{eqn: phi def}
	\phi_b(\alpha) := \left|\sum_{0 \leq m < b}e(\alpha m)\right|
\end{equation}
and
\begin{equation}\label{eqn: def of Phi}
	\Phi_N(\alpha) := \prod_{1 \leq n < N} \phi_b\left(\alpha\left(b^n + b^{2N-n}\right)\right)
\end{equation}
with the convention $\Phi_N(\alpha) = 1$ for $N \leq 1$. The definition of $\Phi_N$ also depends on $b$ but for the sake of brevity we omit it on the left hand side above. 

Finally given integers $n,K,b$ with $K,b \geq 2$, we let 
\begin{equation}\label{eqn: num of compo}
	r(n;K,b) := \#\left\{ (m_1, \ldots, m_K) \in \left([0,b) \cap \mathbb{Z}\right)^K \ : \ m_1 + \cdots + m_K = n\right\} 
\end{equation}
be the number of additive compositions of $n$ into $K$ integers $m_j$, each satisfying $0 \leq m_j < b$.

%%%%%%%%%%%%%%%%%%%%%%%%%%%%%%%%%%%%%%%%%%%%%%%%%%%%%%%%%%%%%%%%%%%%%%%%%%%%%%%%%%%%%%%%%%%%%%%%%%%%%%%%%%%%%%%%%%%%%%%%%%%%%%%%%%%%%%%%%%%%

\section{Technical lemmas}\label{sec: preparatory lemmas}

In this section we go over several technical tools needed in the following two sections. The first two are classical in the theory of uniform distribution of sequences modulo $1$.

\begin{lem}[Koksma-Hlawka]\label{lem: Koksma-Hlawka}
Suppose $f : [0,1] \to \mathbb{R}$ is a bounded integrable function of bounded variation $V(f)$. Then for any real numbers $\alpha_1, \ldots, \alpha_N$ with $N \geq 1$ an integer,
$$
\left|\dfrac{1}{N} \sum_{n=1}^N f(\{\alpha_n\}) - \int_0^1 f(x)dx  \right| \leq V(f)D_N^*(\alpha_1, \ldots, \alpha_N), 
$$
where $D_N^*(\alpha_1, \ldots, \alpha_N)$ is their star-discrepancy as defined in (\ref{eqn: star discrepancy}).
\end{lem}

\begin{proof}
See \cite{Kuipers}. 
\end{proof}

\begin{lem}[Erd\H{o}s-Tur\'{a}n]\label{lem: erdos-turan}
Let $\alpha_1, \ldots, \alpha_N$ be real numbers with $N \geq 1$ an integer. Then
$$
D_N(\alpha_1, \ldots, \alpha_N) \ll \dfrac{1}{H} + \sum_{1 \leq h \leq H} \dfrac{1}{h} \left|\dfrac{1}{N}\sum_{n=1}^N e(h\alpha_n)\right|
$$
for any $H\geq 1$,
where $D_N(\alpha_1, \ldots, \alpha_N)$ is the discrepancy defined in (\ref{eqn: discrepancy}).
\end{lem}

\begin{proof}
	See \cite{Kuipers}. 
\end{proof}

For the following, one says that a function $f : \mathbb{R} \to \mathbb{C}$ is {\em smooth} if $f$ is infinitely differentiable; that is, $f$ has continuous derivatives of all orders. 

\begin{lem}[Smooth exponential sum bounds]\label{lem: smooth exp sums}
	Let $f : \mathbb{R} \to \mathbb{C}$ be a smooth compactly supported function and let $\alpha$ be a real number. Then 	
	\begin{equation}\label{eqn: trivial bound}
		\left|\sum_{n \in \mathbb{Z}}f(n)e(\alpha n) \right| \leq \|f\|_1 + \dfrac{\|f'\|_1}{2}
	\end{equation}
	and
	\begin{equation}\label{eqn: sum by parts bound}
		\left|\sum_{n \in \mathbb{Z}}f(n)e(\alpha n) \right| \leq \dfrac{\|f^{(k)}\|_1}{|2 \sin(\pi \alpha)|^k}
	\end{equation}
	for all natural numbers $k\geq 1$.
\end{lem}	

\begin{proof}
	See for instance Lemma 3.1 in \cite{Tao}. 
\end{proof}	

The following lemma appears, either implicitly or explicitly, on several works on digital functions. See for instance \cite{Bourgain, Col, maud, Maynard, Morgenbesser}. We give a proof for the convenience of the reader.

\begin{lem}[Ergodic-type integral bound]\label{lem: integral of product}
	Let $f : \mathbb{R} \to \mathbb{R}_0^+$ be a bounded, locally integrable, $1$-periodic function. Then for any integer $N \geq 1$,
	$$
	\int_{0}^1 \prod_{0 \leq n < N} f(\alpha b^n)d\alpha \leq \left(\sup_{0 \leq\theta \leq 1}\dfrac{1}{b}\sum_{n(b)} f\left(\dfrac{n + \theta}{b} \right)\right)^N. 
	$$	
\end{lem}

\begin{proof}
	By induction on $N \geq 1$. For $N=1$, we have
	$$
	\int_0^1 f(\alpha)d\alpha = \dfrac{1}{b}\sum_{0 \leq h < b} \int_0^{1} f\left(\dfrac{h + \theta}{b} \right)d\theta
	$$
	and the claim follows. Suppose now that the claim holds for some $N \geq 1$. We have
	\begin{align*}
		\int_0^1 \prod_{0 \leq n < N+1} f(\alpha b^n)d\alpha 
		&=\dfrac{1}{b}\sum_{0 \leq h < b} \int_0^{1} \prod_{0 \leq n < N+1} f\left(\left(\dfrac{h + \theta}{b}\right) b^n\right)d\theta
		\\
		&=
		\int_0^{1} \dfrac{1}{b}\sum_{0 \leq h < b}f\left(\dfrac{h + \theta}{b} \right) \prod_{0 \leq n < N}f\left(\theta b^n\right)d\theta.
	\end{align*}
	The last equality holds since $f$ is $1$-periodic by assumption. The above is
	$$
	\leq \sup_{0 \leq \theta \leq 1} \dfrac{1}{b}\sum_{0 \leq h < b}f\left(\dfrac{h + \theta}{b} \right)	\int_0^{1}\prod_{0 \leq n < N}f\left(\alpha b^n\right)d\alpha
	$$
	and the claim for $N+1$ now follows from the inductive hypothesis.	
\end{proof}

The following lemma gives asymptotics for the number of compositions with some restrictions. We were unable to obtain a reference in the literature and we thus give a proof. If one allows the error term to depend on $n$, one should be able to considerably improve it with a more detailed analysis of the integrals involved. Nevertheless the error term here suffices for our purposes.

\begin{lem}[Compositions with restrictions]\label{lem:comp2}
	Let $n,K,b$ be integers with $K,b \geq 2$ and let $r(n;K,b)$ be as defined in (\ref{eqn: num of compo}). We have 
	$$
	\dfrac{r(n;K,b)}{b^K} =  \sqrt{\dfrac{6}{\pi(b^2-1)K}}\exp\left(-\dfrac{6}{(b^2-1)K}\left(n - \dfrac{(b-1)K}{2}\right)^2\right) + O\left(\dfrac{1}{bK^{3/2}}\right).
	$$
\end{lem}

\begin{proof}
	An application of the identity
	\begin{equation}\label{eqn: orth identity}
	\mathbf{1}_{\ell =0} = \int_{-1/2}^{1/2} e(\ell \alpha)d\alpha
	\end{equation}
	valid for integers $\ell$, yields
	$$
	r(n;K,b) = \int_{-1/2}^{1/2} e(-\alpha n) \left(\sum_{0 \leq m < b}e(\alpha m)\right)^Kd\alpha.
	$$
	In what follows $\delta$ denotes a fixed real number satisfying $1/2 < \delta < 1$. We allow implied constants to depend on $\delta$. To evaluate the integral above we split it into two pieces, one over $|\alpha| \leq \delta/b$ and the other over $\delta/b < |\alpha| \leq 1/2$. We treat each separately. 
	
	When $\alpha \not\in \mathbb{Z}$, the sum inside the brackets above equals
	$$
	\dfrac{e(\alpha b)-1}{e(\alpha)-1} = e\left(\dfrac{b-1}{2}\alpha\right) \dfrac{\sin(\pi \alpha b)}{\sin(\pi \alpha)}.
	$$
	By Euler's product formula,
	$$
	\dfrac{\sin(\pi \alpha b)}{\sin(\pi \alpha)} = b \prod_{m=1}^\infty \left(1 - \dfrac{\alpha^2 b^2}{m^2}\right)\left(1 - \dfrac{\alpha^2}{m^2} \right)^{-1}.
	$$
	If $|\alpha| < 1/b$, the logarithm of the product over $m\geq 1$ equals
	\begin{equation}\label{eqn: log expansion in terms of zeta}
	-\sum_{m\geq 1} \sum_{k\geq 1} \dfrac{(b^{2k}-1)\alpha^{2k}}{m^{2k}k} = - \sum_{k\geq 1} \dfrac{\zeta(2k)}{k}\left(b^{2k}-1\right)\alpha^{2k},
	\end{equation}
	where $\zeta$ is the Riemann zeta function. For $|\alpha| \leq \delta/ b$, the terms with $k\geq 2$ contribute $\ll \alpha^4 b^4$ to the sum. Then it follows 
	$$
	\left(\dfrac{\sin(\pi \alpha b)}{\sin(\pi \alpha)}\right)^K = b^Ke^{-\zeta(2)(b^2-1)K\alpha^2}\left(1 + O\left(\alpha^4  b^4 K\right)\right)
	$$
	for $0 < |\alpha| \leq \delta/b$. 
	Thus, after dividing by $b^K$,
	\begin{align*}
	&\dfrac{1}{b^K}\int_{-\delta/b}^{\delta/b} e(-\alpha n) \left(\sum_{0 \leq m < b} e(\alpha m)\right)^Kd\alpha\\
	 &= \int_{-\delta/b}^{\delta/b} e\left(\dfrac{b-1}{2}\alpha K - \alpha n\right)e^{-\zeta(2)(b^2-1)K\alpha^2}\left(1 + O\left(\alpha^4  b^4 K\right)\right)d\alpha\\
	&=
	\int_{\mathbb{R}} e\left(\dfrac{b-1}{2}\alpha K - \alpha n\right)e^{-\zeta(2)(b^2-1)K\alpha^2} d\alpha \\
	&\qquad 
	+ O\left(b^4 K\int_{\mathbb{R}}e^{-\zeta(2)(b^2-1)K\alpha^2}\alpha^4 d\alpha + \int_{\delta/b}^\infty e^{-\zeta(2)(b^2-1)K\alpha^2}d\alpha\right).
	\end{align*}
	The quantity inside of the $O$-brackets can be shown to be $\ll b^{-1}K^{-3/2}$ after substitutions of variables. By the well-known formula
	$$
	\int_{\mathbb{R}}e^{-at^2}e(-\theta t)dt = \sqrt{\dfrac{\pi}{a}} e^{-\pi^2\theta^2/a} \hspace{2em} (\theta \in \mathbb{R}, \ a > 0) 
	$$
	for the Fourier transform of the Gaussian $e^{-at^2}$ and the fact $\zeta(2) = \pi^2/6$, the integral outside of the $O$-brackets equals
	$$
	\sqrt{\dfrac{6}{\pi (b^2-1)K}} \exp\left(-\dfrac{6}{(b^2-1)K}\left(n - \dfrac{(b-1)K}{2}\right)^2\right).
	$$

	To conclude it suffices to show that
	$$
	\int_{\delta/b < |\alpha| \leq 1/2} e(-\alpha n)\left(\sum_{0 \leq m < b}e(\alpha m)\right)^Kd\alpha \ll \dfrac{b^{K-1}}{K^{3/2}}.
	$$
	The sum inside the brackets is at most $|\csc(\pi \alpha)| \leq 1/2|\alpha|$ (for $|\alpha| \leq 1/2$) in absolute value. Then the absolute value of the integral above is 
	$$\leq 2^{1-K}\int_{\delta/b }^{1/2}\alpha^{-K}d\alpha \leq \dfrac{b^{K-1}}{(2\delta)^{K-1}(K-1)} \ll \dfrac{b^{K-1}}{K^{3/2}}
	$$
	for $\delta > 1/2$ fixed. 
\end{proof}

We also need some trigonometric facts involving the function $\phi_b$ defined as in (\ref{eqn: phi def}). It will be useful to note that
$$
\phi_b(\alpha) = \begin{cases}
	b &\mbox{ if } \alpha \in \mathbb{Z},\\
	|\sin(\pi \alpha b)/\sin(\pi \alpha)| &\mbox{ otherwise.}
\end{cases}
$$
Clearly $\phi_b$ is $1$-periodic and even; thus $\phi_b(\alpha) = \phi_b(\{\alpha\}) = \phi_b(\|\alpha\|)$ for any real $\alpha$.

The following fact, but with $\|\alpha\|$ in the range $\|\alpha\|^2 \leq 6/\pi^2(b^2-1)$, appears in \cite{MR} (see their Lemma 3).

\begin{lem}\label{lem: exponential bound}
	If $\|\alpha\| \leq 1/b$, then
	$$
	\phi_b(\alpha) \leq b \exp\left(-\dfrac{\pi^2 }{6}(b^2-1)\|\alpha\|^2\right).
	$$	
\end{lem}

\begin{proof}
	We have $\phi_b(0) = b$ and $\phi_b(1/b)=0$, whence the result holds for $\|\alpha\| \in \{0, 1/b \}$. For $0 < \|\alpha\| < 1/b$, the claim follows from (\ref{eqn: log expansion in terms of zeta}). 
\end{proof}

\begin{lem}\label{lem: bound for phi}
	If $\alpha$ and $0\leq \delta \leq 2/3b$ are real numbers with $\|\alpha\| \geq \delta$, then $\phi_b(\alpha) \leq \phi_b(\delta)$.
\end{lem}

\begin{proof}
	See Lemma 5 in \cite{MR}.
\end{proof}	

%%%%

\begin{lem}\label{lem: bound for product of pair}
	Let $\alpha,\beta, \gamma$ be real numbers. Then either 
	$$
	\left\|\alpha\left(b^{\beta} + b^{\gamma + 1}\right) \right\| \geq \dfrac{\|\alpha(b^2-1)b^{\gamma}\|}{b+1} 
	$$
	or
	$$
	\left\| \alpha \left(b^{\beta + 1} + b^{\gamma}\right)\right\|\geq \dfrac{\|\alpha(b^2-1)b^{\gamma}\|}{b+1} 
	$$
	(or both).
	Moreover
	$$
	\phi_b\left(\alpha\left(b^{\beta} + b^{\gamma + 1}\right)\right)\phi_b\left(\alpha \left(b^{\beta + 1} + b^{\gamma}\right)\right) \leq b \phi_b \left(\dfrac{\|\alpha(b^2-1)b^\gamma\|}{b+1}\right).
	$$	
\end{lem}

\begin{proof}
	We closely follow the argument in Lemma 6 of Mauduit-Rivat \cite{MR}.
	Making the substitutions $u = \alpha(b^\beta + b^{\gamma+1})$ and $\delta = \|\alpha(b^2-1)b^\gamma\|/(b+1)$, we need to show that either $\|u\| \geq \delta$ or $\|bu - \alpha(b^2-1)b^{\gamma}\| \geq \delta$, and that 
	\begin{equation}\label{eqn: product ineq}
		\phi_b(u)\phi_b(bu - \alpha (b^2-1)b^\gamma) \leq b\phi_b(\delta).
	\end{equation}
	Suppose $\|bu - \alpha(b^2-1)b^\gamma\| < \delta$. The triangle inequality yields
	\begin{align*}
		b\|u\| \geq \|b u\| &= \left\| bu - \alpha(b^2-1)b^\gamma + \alpha(b^2-1)b^\gamma\right\| \\
		&\geq  \left\| \alpha(b^2-1)b^\gamma\right\| - \left\| bu - \alpha(b^2-1)b^\gamma\right\| \\
		&\geq (b+1)\delta - \delta = b\delta.
	\end{align*}
	Thus $\|u\| \geq \delta$. Consequently either $\|u\| \geq \delta$ or $\|bu -\alpha(b^2-1)b^\gamma\|\geq \delta$, while (\ref{eqn: product ineq}) now follows from Lemma \ref{lem: bound for phi}.
\end{proof}	

The following lemma may be useful when considering exponential sums over multisets of the form
$$
\mathcal{S} = \left\{ \left\{\sum_{n=1}^N c_n \alpha_n \ : \ 0 \leq c_1, \ldots, c_N < b \right\}\right\}
$$
for real numbers $\alpha_1, \ldots, \alpha_N$ with a sufficiently small star-discrepancy. This is due to the equality
$$
\left|\sum_{\gamma \in \mathcal{S}}e(\gamma) \right| = \prod_{n=1}^N \phi_b\left(\alpha_n\right). 
$$

\begin{lem}[Weyl product bound]\label{lem: bound for product using KH and ET}
	There exists an absolute constant $A > 0$ such that, for any real numbers $\alpha_1, \ldots, \alpha_N$ with $N \geq 1$ an integer, 
	$$
	\prod_{n =1}^N \phi_b(\alpha_n) \leq \left(\dfrac{2}{D_N^*} \right)^{Ab N D_N^* },
	$$
	where $D_N^* = D_N^*(\alpha_1, \ldots, \alpha_N)$ is their star-discrepancy as defined in (\ref{eqn: star discrepancy}).
\end{lem}

\begin{proof}
	We adapt an argument of Aistleitner et al  \cite{Aistleitner} (see the proof of Theorem 1 there).  
	Let $0 < \epsilon \leq 1/2b$ to be chosen later. For a real number $x$, we define 
	$$
	\phi_{b,\epsilon}(x) := \begin{cases}
	\sin(\pi b \epsilon)/|\sin(\pi x)| &\mbox{ if } |x - \frac{k}{b}| < \epsilon \text{ for some integer $k \not\equiv 0 (b)$}\\
	\phi_b(x) &\mbox{ otherwise.}
	\end{cases}
	$$
	Clearly $\phi_{b,\epsilon}$ is $1$-periodic and continuous satisfying $\phi_{b,\epsilon}(x) \asymp_{b,\epsilon} 1$ and $\phi_{b,\epsilon}(x) \geq \phi_{b}(x)$ for every $x$. In particular
	$$
	\prod_{1 \leq n \leq N} \phi_b(\alpha_n) \leq 
	\prod_{1 \leq n \leq N} \phi_{b,\epsilon}(\alpha_n).
	$$
	Denoting by $P_N$ the product on the right hand side above, taking logarithms and applying the Koksma-Hlawka inequality (Lemma \ref{lem: Koksma-Hlawka}) one has
	\begin{equation}\label{eqn 1}
	\dfrac{\log P_N}{N} = \dfrac{1}{N}\sum_{1 \leq n \leq N} \log \phi_{b,\epsilon}(\alpha_n) \leq \int_0^1 \log\phi_{b,\epsilon}(x)dx +  V(\log \circ \phi_{b,\epsilon}) D^*_N,
\end{equation}
	where $V(\log \circ \phi_{b,\epsilon})$ is the total variation of $\log \circ \phi_{b,\epsilon} $ on $[0,1]$. 
	
	We first bound $V(\log \circ \phi_{b,\epsilon})$. By the definition of $\phi_{b,\epsilon}$, if 
	$$
	x \in \mathcal{S} := [0,1]- \bigcup_{0 < k < b}\left(\frac{k}{b} - \epsilon, \frac{k}{b} + \epsilon\right) = [0,\epsilon]\cup[1-\epsilon,1] \cup \bigcup_{0 \leq k < b}\left[\dfrac{k}{b} + \epsilon, \dfrac{k+1}{b} - \epsilon\right],
	$$ 
	then $\phi_{b,\epsilon}(x) = \phi_b(x) = |\sin(\pi b x)|/\sin(\pi x)$, which we interpret as $b$ if $x \in \{0,1\}$. Over each of the intervals on the right hand side above, $\phi_b$ is continuously differentiable and non-vanishing. Since $\phi_b(x) = \phi_b(\|x\|)$ and $\frac{d}{dt} \log \sin(t) = \cot(t)$, it follows
	\begin{align*}
	&\int_{\mathcal{S}} \left| \dfrac{d}{dx}\log \phi_{b,\epsilon}(x)\right|dx\\
	& \leq 2 \pi \int_0^\epsilon \left|b\cot(\pi b x) - \cot(\pi x)\right|dx + \pi  \sum_{0 \leq k < b}\int_{\frac{k}{b} + \epsilon}^{\frac{k+1}{b}-\epsilon}\left(b |\cot(\pi b x)| + |\cot(\pi x)|\right)dx.
\end{align*}  
Since $\cot(x) = x^{-1} + O(x)$ for $0 < |x| \leq \pi/2$ and $\epsilon \leq 1/2b$ by assumption, the integral over $[0,\epsilon]$ is $\ll 1$. Substituting variables and using $|\cot(\pi t)| = \cot(\pi \|t\|)$, we also have
$$
b\int_{\frac{k}{b} + \epsilon}^{\frac{k+1}{b} - \epsilon}|\cot(\pi b x)|dx = \int_{k + b\epsilon}^{k+1 - b\epsilon} |\cot(\pi x)|dx = 2\int_{b\epsilon}^{1/2}\cot(\pi x)dx \ll \log(1/b\epsilon). 
$$
Thus
$$
\int_{\mathcal{S}} \left| \dfrac{d}{dx}\log \phi_{b,\epsilon}(x)\right|dx \ll b\log(1/b\epsilon) + \pi \sum_{0 \leq k < b}\int_{\frac{k}{b} + \epsilon}^{\frac{k+1}{b} - \epsilon}|\cot(\pi x)|dx.
$$ 
For $x \in [0,1] - \mathcal{S}$, we have $\phi_{b,\epsilon}(x) = \sin(\pi b \epsilon)/\sin(\pi x )$; hence $|\frac{d}{dx} \log \phi_{b,\epsilon}(x)| = \pi |\cot(\pi x)|$. Thus
$$
\int_{[0,1] - \mathcal{S}} \left| \dfrac{d}{dx}\log \phi_{b,\epsilon}(x)\right|dx = \pi \sum_{0 < k < b}\int_{\frac{k}{b} - \epsilon}^{\frac{k}{b} + \epsilon}|\cot(\pi x)|dx
$$
and so
\begin{align}
&V\left(\log \circ \phi_{b,\epsilon}\right)\nonumber\\ 
&\ll b\log(1/b\epsilon) + \pi \sum_{0 \leq k < b}\int_{\frac{k}{b} + \epsilon}^{\frac{k+1}{b} - \epsilon}|\cot(\pi x)|dx + \pi \sum_{0 < k < b}\int_{\frac{k}{b} - \epsilon}^{\frac{k}{b} + \epsilon}|\cot(\pi x)|dx\nonumber\\
&=b\log(1/b\epsilon) + \pi \int_\epsilon^{1-\epsilon} |\cot(\pi x)|dx = b\log(1/b\epsilon) + 2\pi \int_\epsilon^{1/2} \cot(\pi x)dx \nonumber\\
&\ll b\log(1/b\epsilon). \label{eqn 2}
\end{align}

	With regards to the integral $\int_0^1\log \phi_{b,\epsilon}(x)dx$, splitting it similarly according to the definition of $\phi_{b,\epsilon}$, using $\phi_b(x) \leq b$, the additivity of the logarithm, and ignoring the terms $\log \sin(\pi b \epsilon)$ (these are non-positive) one derives 
	\begin{align}
	\int_0^1 \log \phi_{b,\epsilon}(x)dx &\leq 2\epsilon \log b +  2\int_{\epsilon}^{1/2} \log \csc(\pi x)dx + \sum_{0 \leq k < b} \int_{\frac{k}{b} + \epsilon}^{\frac{k+1}{b} - \epsilon} \log |\sin(\pi b x)|dx \nonumber\\
	&=2\epsilon \log b +  2\int_{\epsilon}^{b\epsilon} \log \csc(\pi x)dx \nonumber\\
	&\ll b\epsilon \log(1/b \epsilon).\label{eqn 3} 
	\end{align}
	The equality above follows after substituting $bx$ with $u$ and using
	$$
	\int_{k+b\epsilon}^{k+1 - b\epsilon}\log |\sin(\pi u)|du  = \int_{b\epsilon}^{1-b\epsilon} \log |\sin (\pi u)|du = -2 \int_{b\epsilon}^{1/2} \log \csc(\pi u)du.
	$$
	Now the result follows from (\ref{eqn 1}), (\ref{eqn 2}), (\ref{eqn 3}) if we let $\epsilon = D_N^*/2b $ using the fact $0 < 1/2N \leq D_N^* \leq 1$ (see for instance Theorems 1.2 and 1.3 in Section 2 of \cite{Kuipers} for this).
\end{proof}

Our final tool is the following version of Vinogradov's lemma. 

\begin{lem}[Vinogradov-type lemma]\label{lem: vinogradov}
	Let $A,B, \theta$ be real numbers, $A,B > 0$, and let $q\geq 2$ be an integer. Then
	$$
	\sum_{n(q)} \min\left(A, B\csc^2\left(\pi \dfrac{n + \theta}{q}\right)\right) \leq \min\left(A, B\csc^2\left(\dfrac{\pi \|\theta\|}{q}\right)\right) + \left(1 - \dfrac{4}{\pi^2}\right)Bq^2.
	$$
\end{lem}

\begin{proof}
	We can write $\theta = m + r$ for some integer $m$ and real number $-1/2 < r \leq 1/2$. In fact $|r| = \|\theta\|$. Since $m + n$ uniquely covers each residue class modulo $q$ when so does $n$, the sum equals
	$$
	S := \sum_{0 \leq n < q} \min\left(A, B\csc^2\left(\pi \dfrac{n + r}{q}\right)\right).
	$$
	If $r < 0$, we may replace $n$ with $-n$ in the summand. Thus without loss of generality we may assume that $r \geq 0$. Now note
	$$
	S \leq \min\left(A, B\csc^2\left( \dfrac{\pi r}{q}\right)\right) + B H_q(r), 
	$$
	where for $0 \leq r \leq 1/2$, 
	$$
	H_q(r) :=  \sum_{1 \leq n < q} \csc^2\left(\pi \dfrac{n + r}{q} \right). 
	$$
	Since $x \mapsto \csc^2(\pi x)$ is convex in $0 < x < 1$, so is $H_q(r)$ in $0 \leq r \leq 1/2$. From the well-known identity
	$$
	\pi^2\csc^2(\pi x) = \sum_{n \in \mathbb{Z}} \dfrac{1}{(n + x)^2}
	$$
	for $x \not\in \mathbb{Z}$,
	one derives (see for instance \cite{Hofbauer})
	$$
	H_q(r) = \begin{cases}
	(q^2-1)/3 &\mbox{ if } r=0, \\
	q^2 \csc^2(\pi r) - \csc^2(\pi r/q) &\mbox{ if } 0 < r \leq 1/2.
	\end{cases}
	$$
	One can show that $H_q(0)\leq H_q(1/2)$ (say by using the inequality $\sin(\pi x) \geq 2^{3/2} x$ valid for $0 \leq x \leq 1/4$) and thus $H_q(r) \leq H_q(1/2)$ for $0 \leq r \leq 1/2$ by the convexity of $H_q(r)$ in $[0,1/2]$. We also have $H_q(1/2) = q^2 - \csc^2(\pi/2q) \leq q^2(1 - 4/\pi^2)$ since $\sin(x) \leq x$ for $x \geq 0$. The result now follows.
\end{proof}

\section{Exponential sums over $\mathscr{P}_b^0(x)$}\label{sec: exp sums}

The current state of knowledge on exponential sums over palindromes, specifically on bounds for these, is rather limited. To our understanding, the literature on this topic is thus far comprised of the works of Banks-Hart-Sakata \cite{Banks-Hart-Sakata}, Banks-Shparlinski \cite{Banks-Shparlinski} and Col \cite{Col}.
Here we give a brief exposition of what is known, as well as add slightly to the literature via the results in Proposition \ref{prop: product bound for modulus p} and Corollary \ref{cor: bound using product of Bourgain sequence} for special prime moduli. These are consequences of a result of Bourgain \cite{Bourg} giving upper bounds for certain types of exponential sums over finite fields with prime order. However these are not employed by us in our proofs of Theorems \ref{thm: at most 6 primes} and \ref{thm: equidistribution}. As such, Proposition \ref{prop: product bound for modulus p} and Corollary \ref{cor: bound using product of Bourgain sequence} may be regarded as addenda to the work, stated here in the interests of expanding the subject.

For the sake of simplicity, and indeed, for our purposes, we focus exclusively on exponential sums over $\mathscr{P}_b^0(x)$; that is, over $b$-palindromes, at most $x$ in size, with an odd number of digits in their $b$-adic expansion. Sums over palindromes with an even number of digits are rather similar in nature and all results here extend naturally, albeit with minor differences, to these.

The starting point, just as that in the works \cite{Banks-Hart-Sakata}, \cite{Col}, is that such sums can essentially be decomposed into linear combinations of objects enjoying rather useful multiplicative and algebraic structures.
These properties are crucial in the upcoming sections. In this regard we have the following lemma, implicit in the works \cite{Banks-Hart-Sakata}, \cite{Col}. We give a proof for the convenience of the reader.

\begin{lem}\label{lem: bound for exp sum of pals}
	For any $\alpha,x \in \mathbb{R}$, $x \geq 1$,
	$$
	\left| \sum_{n \in \mathscr{P}_b^0(x)} e(\alpha n)\right| \leq b^2 \sum_{0 \leq N \leq \frac{1}{2}\log_b x} \sum_{0 \leq M \leq N} \Phi_{M}\left(\alpha b^{N-M}\right),
	$$
	where $\Phi_M$ is as defined in (\ref{eqn: def of Phi}).
\end{lem}

\begin{proof}
	From the definition of $\mathscr{P}_b^0(x)$ and the triangle inequality,
	$$
	\left| \sum_{n \in \mathscr{P}^0_b(x)} e(\alpha n)\right| \leq \sum_{0 \leq N \leq \frac{1}{2} \log_b x} \left|\sum_{\substack{n \in \Pi_b(2N) \\ n\leq x}}e(\alpha n)\right|, 
	$$
	where
	\begin{equation}\label{eqn: def of Pi}
		\Pi_b(2N) := \left\{n \in \mathscr{P}_b \ : \ \lfloor \log_b n \rfloor = 2N\right\} = \mathscr{P}_b \cap [b^{2N}, b^{2N + 1}).
	\end{equation}
	The sum inside the vertical brackets equals
	$$
	\sum_{\substack{n \in \Pi_b(2N) \\ n\leq y}}e(\alpha n),
	$$
	where $y$ is the largest palindrome in $\Pi_b(2N)$ satisfying $y \leq x$. If no such palindrome $y$ exists, the sum is trivially zero. Otherwise we may write 
	$$
	y = 	 y_N b^N + \sum_{0 \leq j < N} y_j \left(b^j + b^{2N-j}\right)
	$$ for some digits $0 \leq y_j < b$ with $y_0 > 0$. 
	
	We now seek to establish the inequality 
	\begin{equation}\label{eqn: palind exp sum}
		\left|\sum_{\substack{n \in \Pi_b(2N) \\ n\leq y } } e(\alpha n)\right| \leq b^2 \sum_{0 \leq M \leq N} \Phi_{M}\left(\alpha b^{N-M}\right).
	\end{equation}  
	To this end, we first define for an integer $\lambda \geq 0$,	
	$$
	\Psi_\lambda(\alpha) := \sum_{0 \leq c_N < b} e\left(\alpha c_N b^N\right) \prod_{\lambda < n < N} \sum_{0 \leq c < b} e\left(\alpha c \left(b^n + b^{2N-n}\right)\right).
	$$ 
	We have
	\begin{align*}
		&\sum_{\substack{n \in \Pi_b(2N) \\ n\leq y } } e(\alpha n) \\
		&= \Psi_0(\alpha) \sum_{1 \leq c < y_0} e\left(\alpha c\left(1 + b^{2N}\right)\right) \\
		&
		+ \Psi_1(\alpha) e\left(\alpha y_0\left(1 + b^{2N}\right)\right) 
		\sum_{0 \leq c < y_1}e\left(\alpha c \left(b + b^{2N-1}\right)\right) \\
		&+ 
		\Psi_2(\alpha) e\left(\alpha y_0\left(1 + b^{2N}\right) + \alpha y_1 \left(b + b^{2N-1}\right)\right) \sum_{0 \leq c < y_2} e\left(\alpha c\left(b^2 + b^{2N-2}\right)\right) \\
		&+\cdots\\
		&+
		\Psi_{N-1}(\alpha) e\left(\alpha\sum_{0 \leq j < N-1} y_j \left( b^j + b^{2N-j}\right)\right)\sum_{0 \leq c < y_{N-1}} e\left(\alpha c\left(b^{N-1} + b^{2N-(N-1)}\right)\right)\\
		&+
		e\left(\alpha\sum_{0 \leq j < N} y_j \left( b^j + b^{2N-j}\right)\right) \sum_{0 \leq c \leq y_N } e\left(\alpha c b^N\right). 
	\end{align*}
	Now the result follows when we note that 
	$$
	\prod_{\lambda < n < N} \sum_{0 \leq c < b} e\left(\alpha c \left(b^n + b^{2N-n}\right)\right) = \prod_{1 \leq n < N - \lambda} \sum_{0 \leq c < b} e\left(\alpha b^{\lambda}  c \left(b^n + b^{2(N-\lambda)-n}\right)\right) 
	$$	
	and take absolute values. 
\end{proof}	

From the lemma above it is clear that to study upper bounds for sums such as $\sum_{n \in \mathscr{P}_b^0(x)}e(\alpha n)$, it is sufficient to study the products $\Phi_N$ or, more generally, products such as
\begin{equation}\label{def: P_M}
P_M(\alpha) := \prod_{1 \leq n \leq M} \phi_b\left(\alpha\left(b^n + b^{2N-n}\right)\right)
\end{equation}
with $M \leq 2N$.
Of course the product's definition depends also on $b,N$, but for the sake of brevity we omit these on the left hand side above. 

Let us now state Col's result in our context. This will be needed later on. We give a proof somewhat different to that of Col \cite{Col} in its initial steps, although it ultimately boils to considering lower bounds for sums such as $\sum_{n \leq M}\|ab^n/q\|^2$, just as in Col's argument.

\begin{prop}[Col \citelist{\cite{Col}, Corollary 4}]\label{prop: Linfty bound}
	Let $h,q,k,M,N$ be integers with $q \geq 2$, $(q,h(b^3-b))=1$ and $0 \leq M \leq 2N$. Then for $P_M$ defined as in (\ref{def: P_M}),
	$$
	P_M\left(\dfrac{h}{q} + \dfrac{k}{b^3-b}\right) \ll_b b^{M} \exp\left(-\sigma_\infty(b) \dfrac{M}{\log q}\right)
	$$
	uniformly in $h,k,N$, where $\sigma_\infty(b) > 0$ is some value depending only on $b$.
\end{prop}

\begin{proof}
	We may assume that $M \geq \log_b q$ as otherwise the statement is trivial. 
	Now for the sake of brevity let $\alpha := \frac{h}{q} + \frac{k}{b^3-b}$. 
Grouping the factors in the product by pairs of adjacent factors, we have
	\begin{align*}
	P_M(\alpha) \leq b \prod_{1 \leq n <M} \sqrt{\phi_b\left(\alpha \left(b^n + b^{2N-n}\right)\right)\phi_b\left(\alpha\left(b^{n+1} + b^{2N - n-1}\right)\right)}.
	\end{align*}
	By Lemma \ref{lem: bound for product of pair} with $\beta = 2N - n - 1$ and $\gamma = n$, 
	$$
	\phi_b\left(\alpha \left(b^n + b^{2N-n}\right)\right)\phi_b\left(\alpha\left(b^{n+1} + b^{2N - n-1}\right)\right) \leq b\phi_b \left( \dfrac{\|\alpha(b^2-1)b^{n}\|}{b+1}\right).
	$$
	Thus
	$$
	P_M(\alpha) \leq b^{\frac{M-1}{2} + 1} \sqrt{\prod_{1 \leq n < M} \phi_b \left(\dfrac{\|\alpha (b^2-1)b^n\|}{b+1}\right)}.
	$$
	By Lemma \ref{lem: exponential bound} and the $1$-periodicity of $\|\cdot\|$, the above is
	\begin{align*}
		\ll b^M \exp\left(- \dfrac{\pi^2(b-1)}{12(b+1)}\sum_{1 \leq n \leq M} \left\| \dfrac{h(b^2-1)b^n}{q}\right\|^2\right)\end{align*}
	uniformly in $k,N$. Since $q \geq 2$ and $(q,h(b^3-b))=1$ by assumption, each term in the sum above is $\geq 1/q^2$. In particular the result holds if $q < b$. Consider now the case when $q \geq b$. Here we argue as done by Maynard \cite{Maynard} in Lemma 10.1 there. 
	Dividing the sum above into subsums over segments of length $\lfloor \log_b q\rfloor$, we claim it is
	$$
	\geq \left\lfloor \dfrac{M}{\log_b 	q}\right\rfloor \min_{(a,q)=1}\sum_{1 \leq n \leq \log_b q} \left\| \dfrac{a b^n}{q}\right\|^2.
	$$
	Indeed, with the possible exception of one subsum (coming from the tail with $n$ within a distance $ < \lfloor\log_b q\rfloor$ from $M$, and lower-bounded by zero) each subsum can be written as
	$$
	\sum_{1 \leq n \leq \log_b q} \left\| \dfrac{h(b^2-1)b^{m+n}}{q} \right\|^2 
	$$
	for some integer $m \geq 0$. There are at least $\lfloor M/\log_b q\rfloor$ subsums of this type. Since $(q,h(b^3-b))=1$ by assumption, the claim follows. 
	Now since $q \geq 2$ and $(q,ab)=1$, we have $\|ab^n/q\| \geq 1/q$ for each $n$. Moreover if $\|ab^n/q\| \leq 1/2b$, then $\|ab^{n+1}/q\| = b\|ab^n/q\|$. Consequently there exists an integer $1 \leq n \leq \log_b q$ for which $\|ab^n/q\| \geq 1/2b^2$. Hence 
	$$
	\sum_{1 \leq n \leq \log_b q} \left\| \dfrac{a b^n}{q}\right\|^2 \geq \dfrac{1}{4b^4}
	$$
	and the result follows. 
\end{proof}

The bound due to Col \cite{Col} in Proposition \ref{prop: Linfty bound} is non-trivial in the range $q \ll e^{o(M)}$ and is rather general in the sense of the lack of restrictions on the arithmetic shape of $q$. As such, it considerably improved over the corresponding general bound due to Banks-Hart-Sakata \cite{Banks-Hart-Sakata} (see Lemma 3.2 there) non-trivial only in the range $q = o(\sqrt{M})$ (strictly speaking, the bounds of Banks-Hart-Sakata were essentially given for products such as $\Phi_N$ and not the product over $1 \leq n \leq M$ in Proposition \ref{prop: Linfty bound}, but their work easily extends to this).

In what follows $\text{ord}_q(b)$ denotes the multiplicative order of $b$ modulo $q$ (whenever $(q,b)=1$). 
In the range $q \ll M^2$ with $P^-(q) > b$ satisfying $\tau(q)\sqrt{q} \ll \text{ord}_q(b) \ll M$ (with suitable implied constants) the arguments of Banks-Hart-Sakata \cite{Banks-Hart-Sakata} (see Lemma 3.1 there) yield power-saving bounds of the form $P_M(h/q)\ll b^{\delta M}$ with $0 < \delta < 1$ fixed or depending only on $b$. Although applicable only on the much smaller range $q \ll M^2$, this is superior in strength over that in Proposition \ref{prop: Linfty bound} for such $q$. Indeed, the latter attains (at best) power-savings only for $q \ll 1$. In few words, their argument can be explained as follows:

Assume $(h,q) = 1$. Squaring $P_M(h/q)$, applying the inequality of the arithmetic and geometric means, expanding the square and switching orders of summation, all in turn, one has
\begin{align}
P_M^{2/M}(h/q) &\leq  \sum_{0 \leq c_1, c_2 < b}\dfrac{1}{M}\sum_{1 \leq n \leq M} e_q\left(h(c_1-c_2)\left(b^n + b^{2N-n}\right)\right) \label{eqn: AGM bound}\\
&= 
 b + \sum_{0 \leq c_1 \neq c_2 < b}\dfrac{1}{M}\sum_{1 \leq n \leq M} e_q\left(h(c_1-c_2)\left(b^n + b^{2N-n}\right)\right).\nonumber
\end{align}
If we impose the restriction $P^-(q) > b$, then $(c_1 - c_2, q)=1$ for $0 \leq c_1 \neq c_2 < b$. In this case
\begin{equation}\label{eqn: bound for P_M}
P_M^{2/M}(h/q) \leq b + \dfrac{b^2}{M} \max_{a,k \in (\mathbb{Z}/q\mathbb{Z})^\times} S_q(M,a,k), 
\end{equation}
where
\begin{equation}\label{def: S}
S_q(M,a,k) := \left|\sum_{1 \leq n \leq M} e_q\left(ab^n + k\bar{b}^{n}\right)\right|.
\end{equation}
Here $\overline{b}$ denotes the multiplicative inverse of $b$ modulo $q$. We note that sums such as these also appear for instance in Banks-Shparlinski \cite{Banks-Shparlinski}, Bourgain \cite{Bourg} (see Theorem \ref{thm: Bourgain} below) Humphries \cite{Humphries}, Ostafe-Shparlinski-Voloch \cite{Ostafe} and Popova \cite{Popova}. 

Using well-known bounds for twisted Kloosterman sums, one derives (see Lemma 2.1 in \cite{Banks-Hart-Sakata})
$$
S_q(M,a,k) \leq \dfrac{M \tau(q) \sqrt{q} }{\text{ord}_q(b)} + \text{ord}_q(b).
$$
For this to be non-trivial, it is necessary that $\tau(q)\sqrt{q} \ll \text{ord}_q(b) \ll M$; hence their assumptions on $q$. One thus obtains
$$
P_M^{2/M}(h/q) \leq b + b^2 \left(\dfrac{\tau(q) \sqrt{q}}{\text{ord}_q(b)} + \dfrac{\text{ord}_q(b)}{M}\right).
$$

Banks-Shparlinski \cite{Banks-Shparlinski} also studied the sum $S_q(M,a,k)$ after the work of Banks-Hart-Sakata \cite{Banks-Hart-Sakata}
and obtained a bound for $P_M(h/p)$, with $p$ prime, essentially comparable to Col's result in Proposition \ref{prop: Linfty bound} for $q=p \ll M^2/\log^4 M$ with no restrictions on $\text{ord}_p(b)$; see their Theorem 6.

From the argument above, it is clear that improving bounds for $S_q(M,a,k)$ and/or enlarging the range of applicability of such bounds, may yield stronger results. However due to the presence of the $b$ term in (\ref{eqn: bound for P_M}) coming from the diagonal cases $c_1=c_2$ in (\ref{eqn: AGM bound}), one can never do better than $P_M \ll b^{M/2}$ via this argument, regardless of the quality of the bounds available for $S_q(M,a,k)$.  

Soon after the works of Banks-Hart-Sakata \cite{Banks-Hart-Sakata} and Banks-Shparlinski \cite{Banks-Shparlinski}, Bourgain \cite{Bourg} obtained the following result. First given a prime $p$, let $\mathbb{F}_p$ be the finite field with $p$ elements and let $\mathbb{F}_p^*$ be its subset of non-zero elements, endowed with the usual properties of $\mathbb{F}_p$. Given $\theta \in \mathbb{F}_p^*$, we use the notation $\text{ord}(\theta)$ to denote its multiplicative order. 

\begin{thm}[Bourgain \citelist{\cite{Bourg}, Theorem 2}]\label{thm: Bourgain}
	Let $\epsilon > 0$, let $p$ be a prime number and let $\theta_1, \ldots, \theta_r \in \mathbb{F}_p^*$ satisfying $\text{ord}(\theta_j) > p^\epsilon$ and $\text{ord}(\theta_i \theta_j^{-1}) > p^{\epsilon}$ for each $1 \leq i\neq j\leq r$. Then for any integer $N > p^\epsilon$ and any $a_1, \ldots, a_r \in \mathbb{F}_p^*$, 
	$$
	\left|\sum_{n=1}^N e_p\left(\sum_{j=1}^r a_j \theta_j^n \right)\right| < \dfrac{N}{p^\delta},
	$$
	where $0 < \delta = \delta(\epsilon) < 1$ is some value depending only on $\epsilon$. 
\end{thm}

We recognize that the sum, in Theorem \ref{thm: Bourgain}, with $r=2, a_1=a, a_2=k, \theta_1 = b, \theta_2 = \bar{b}$, $N=M$, is precisely $S_p(M,a,k)$. Thus we have the following. 

\begin{cor}\label{cor: cor of Bourgain}
	Let $M $ be an integer. Then for any prime $p$ satisfying $(p,b)=1$ and $p^{\epsilon} < \min(M, \text{ord}_p(b^2))$ for some $\epsilon > 0$, we have
$$
\max_{a,k \in (\mathbb{Z}/p\mathbb{Z})^\times} S_p(M,a,k) < \dfrac{M}{p^{\delta(\epsilon)}},
$$
where $\delta(\epsilon) > 0$ depends only on $\epsilon$, and $S_p(M,a,k)$ is defined as in (\ref{def: S}). 
\end{cor}

\begin{proof}
	Suffices to note $\text{ord}_p(b^{\pm 1}) \geq \text{ord}_p(b^{\pm 2}) > p^\epsilon$, the last holding by assumption.
\end{proof}

From this and (\ref{eqn: bound for P_M}) we may then obtain, under the same assumptions of Corollary \ref{cor: cor of Bourgain} with $p > b$, 
$$
P_M\left(\dfrac{h}{p}\right) \leq b^M \left(\dfrac{1}{b} + \dfrac{1}{p^{\delta}}\right)^{M/2}.
$$
One may improve this considerably, in Corollary \ref{cor: bound using product of Bourgain sequence} below, by combining use of Lemma \ref{lem: bound for product using KH and ET}, the Erd\H{o}s-T\'{u}ran inequality and Bourgain's bound. More generally we have the following.

\begin{prop}\label{prop: product bound for modulus p}
With the same assumptions of Theorem \ref{thm: Bourgain},
$$
\prod_{1 \leq n \leq N} \phi_b\left(\dfrac{1}{p} \sum_{j=1}^r a_j \theta_j^n \right) \leq \exp\left( A(\epsilon) b\dfrac{ N  }{p^{\delta(\epsilon)}}\right),
$$
where $A(\epsilon), \delta(\epsilon) > 0$ are some values depending only on $\epsilon$.
\end{prop}

\begin{proof}
	We may view each $\sum_{j=1}^r a_j \theta_j^n$ as a real number in the natural way.
	We may also assume that $p \gg_\epsilon 1$ is arbitrarily large. 
	Now for each integer $1 \leq n \leq N$, let 
	$$
	\alpha_n = \dfrac{1}{p} \sum_{j=1}^r a_j \theta_j^n.
	$$
	By Lemma \ref{lem: bound for product using KH and ET}, 
	$$
	\prod_{1 \leq n \leq N} \phi_b(\alpha_n) \leq \exp\left(AbND_N^* \log\left(\dfrac{2}{D_N^*}\right)\right) 
	$$
	for some absolute constant $A > 0$, where $D_N^*$ is the star-discrepancy of $(\alpha_n)_{n=1}^N$. By the Erd\H{o}s-Tur\'{a}n inequality (Lemma \ref{lem: erdos-turan}) and the fact $D_N^* \leq D_N$, we have
	$$
	D_N^* \ll \dfrac{1}{H} + \sum_{1\leq h \leq H} \dfrac{1}{h} \left|\dfrac{1}{N} \sum_{1 \leq n \leq N}e(h\alpha_n)\right| 
	$$
	for any $H \geq 2$. The contribution of the terms with $p \mid h$ to the series above is trivially $\ll \log(H)/p$. By Bourgain's bound in Theorem \ref{thm: Bourgain}, the terms with $p \nmid h$ contribute $\ll \log(H)/p^{\delta}$ to the sum, where $0 < \delta = \delta(\epsilon) < 1 $ is some value depending only on $\epsilon$. Thus
	$$
	D_N^* \ll \dfrac{1}{H} + \dfrac{\log H}{p^\delta} \ll \dfrac{\log p}{p^{\delta}} 
	$$
	after letting $H = p$. Since the function $x \mapsto x\log(2/x)$ is increasing on $(0, 2/e)$, then for $p \gg_\epsilon 1$ sufficiently large (as assumed)
	$$
	D_N^* \log(2/D_N^*) \ll \dfrac{\log^2 p}{p^{\delta}} \ll_\delta p^{-\delta/2}
	$$
	and the result follows.
\end{proof}

Specializing to the palindromes we have the following immediate consequence.

\begin{cor}\label{cor: bound using product of Bourgain sequence}
	With the same assumptions of Corollary \ref{cor: cor of Bourgain} and $M\leq 2N$ there, we have
	$$
	\max_{(h,p)=1}\prod_{1 \leq n \leq M} \phi_b\left(\dfrac{h}{p} \left(b^n + b^{2N-n}\right)\right) \leq \exp\left( A(\epsilon)b\dfrac{ M  }{p^{\delta(\epsilon)}}\right),
	$$
	where $A(\epsilon), \delta(\epsilon) > 0$ are some values depending only on $\epsilon$.
	\end{cor}

\section{Bounding moments of $\Phi_N$}\label{sec: moments}

Here we prove Proposition \ref{prop:2K-th moment2} bounding moments of $\Phi_N$.

\begin{prop}[$2K$-th moment]\label{prop:2K-th moment2}
	For any integers $N,K,b \geq 2$, we have
\begin{equation}\label{eqn: 2K integral}
\int_{0}^1 \Phi_N^{2K}(\alpha) d\alpha \leq b^{2(K-1)N + 2} \left(1 + O\left(\dfrac{1}{\sqrt{K}} + \dfrac{b^2}{K}\right)\right)^{2N}.
\end{equation}
	
If $\alpha_1, \ldots, \alpha_R \in \mathbb{R}/\mathbb{Z}$ are $\delta$-spaced for some $0 < \delta \leq 1/2$, then
\begin{equation}\label{eqn: 2K sieve sum}
\sum_{r=1}^R \Phi_N^{2K}\left(\alpha_r\right) \leq \left(\delta^{-1} + Kb^{2N}\right)b^{2(K-1)N + 2}\left(1 + O\left(\dfrac{1}{\sqrt{K}} + \dfrac{b^2}{K}\right)\right)^{2N}.
\end{equation}
In particular for any $Q \geq 1$ and uniformly in $\beta \in \mathbb{R/\mathbb{Z}}$,
	\begin{equation}\label{eqn: sieve fractions}
	\sum_{q \leq Q} \sum_{h(q)}^* \Phi_N^{2K}\left(\dfrac{h}{q} + \beta\right) \leq \left(Q^2 + Kb^{2N}\right)b^{2(K-1)N + 2} \left(1 + O\left(\dfrac{1}{\sqrt{K}} + \dfrac{b^2}{K}\right)\right)^{2N}.
	\end{equation}
\end{prop}

\begin{proof}
	%[{\bf Proof of Proposition \ref{prop:2K-th moment2}}]
	We first note that (\ref{eqn: sieve fractions}) follows from (\ref{eqn: 2K sieve sum}) since the points $h/q + \beta$ with $1 \leq h \leq q \leq Q$ and $(h,q)=1$ are $Q^{-2}$-spaced modulo $1$. Indeed, for any two distinct such fractions $h_1/q_1, h_2/q_2$,
	$$
	\left\| \dfrac{h_1}{q_1} - \dfrac{h_2}{q_2}\right\| \geq \dfrac{1}{q_1q_2} \geq \dfrac{1}{Q^2}.
	$$

	Let us now set $\psi_b(\alpha) := \sum_{0 \leq m < b}e(\alpha m)$. By definition,
	$$
	\Phi_N^{2K}(\alpha) = \left|\prod_{1 \leq n < N}\psi_b^K\left(\alpha \left(b^n + b^{2N-n}\right)\right)\right|^2.
	$$
	We have
	$$
	\prod_{1 \leq n < N} \psi_b\left(\alpha \left(b^n + b^{2N-n}\right)\right) = \sum_{0 \leq c_1, \ldots,c_{N-1} < b }e\left(\alpha  \sum_{1 \leq n < N}c_n \left(b^n + b^{2N-n}\right)\right)
	$$
	and one observes
	\begin{align*}
	&\prod_{1 \leq n < N} \psi_b^K\left(\alpha \left(b^n + b^{2N-n}\right)\right)\\
	&= \sum_{0 \leq v_1, \ldots, v_{N-1} \leq (b-1)K} \left(\prod_{1 \leq n < N} r(v_n; K,b) \right)e\left(\alpha \sum_{1 \leq m < N} v_m \left(b^m + b^{2N-m}\right)\right),
	\end{align*}
	where $r(v;K,b)$ is defined as in (\ref{eqn: num of compo}). We can rewrite this as 
	$$
	\prod_{1 \leq n < N} \psi_b^K\left(\alpha \left(b^n + b^{2N-n}\right)\right) = \sum_{0 \leq \ell \leq Kb^{2N}} a_\ell e(\alpha \ell),
	$$
	where
	$$
	a_\ell := \sum_{\substack{0 \leq v_1, \ldots, v_{N-1} \leq (b-1)K \\ \sum_{1 \leq m < N} v_m(b^m + b^{2N-m}) = \ell}}\prod_{1 \leq n <N} r(v_n;K,b).
	$$
	Thus 
	$$
	\Phi_N^{2K}\left(\alpha \right) = \left|\sum_{0\leq \ell \leq Kb^{2N}}a_\ell e(\alpha \ell)\right|^2.
	$$
	Parseval's identity then gives
	$$
	\int_0^1 \Phi_N^{2K}\left(\alpha \right)d\alpha = \sum_{0\leq \ell \leq Kb^{2N}} a_\ell^2
	$$
	while the large sieve inequality (Lemma \ref{lem: large sieve}) yields
	$$
	\sum_{r=1}^R \Phi_N^{2K}\left(\alpha_r \right)\leq \left(\delta^{-1} + Kb^{2N}\right)\sum_{0\leq \ell \leq Kb^{2N}} a_\ell^2
	$$
	for any $\delta$-spaced points $\alpha_1, \ldots, \alpha_R \in \mathbb{R}/\mathbb{Z}$. Thus to prove the proposition we must show that
	\begin{equation}\label{eqn: bound for sum of squares}
	\sum_{0\leq \ell \leq Kb^{2N}} a_\ell^2 \leq b^{2(K-1)N + 2} \left(1 + O\left(\dfrac{1}{\sqrt{K}} + \dfrac{b^2}{K}\right)\right)^{2N}.
	\end{equation}

	Expanding the square and switching orders of summation, we have
	$$
	\sum_{0 \leq \ell \leq Kb^{2N}}a_\ell^2 = \sum_{\substack{0\leq u_1, \ldots, u_{N-1} \leq (b-1)K \\ 0 \leq v_1, \ldots, v_{N-1} \leq (b-1)K \\ \sum_{1 \leq m < N}(u_m - v_m)(b^m + b^{2N-m}) = 0}} \prod_{1 \leq n < N} r(u_n; K,b)r(v_n;K,b).
	$$
	It will be convenient later on to introduce extra variables $u_N,v_N$ to the sum. To this end, we claim that 
	\begin{align*}
	&\sum_{\substack{0\leq u_1, \ldots, u_{N-1} \leq (b-1)K \\ 0 \leq v_1, \ldots, v_{N-1} \leq (b-1)K \\ \sum_{1 \leq m < N}(u_m - v_m)(b^m + b^{2N-m}) = 0}} \prod_{1 \leq n < N} r(u_n; K,b)r(v_n;K,b)\\
	& 
	\leq \sum_{\substack{0\leq u_1, \ldots, u_{N} \leq (b-1)K \\ 0 \leq v_1, \ldots, v_{N} \leq (b-1)K \\ \sum_{1 \leq m \leq N}(u_m - v_m)(b^m + b^{2N-m}) = 0}} \prod_{1 \leq n \leq N} r(u_n; K,b)r(v_n;K,b).
	\end{align*}
	Indeed, any solution $(u_1', \ldots, u_{N-1}') \times (v_1', \ldots, v_{N-1}')$ to the equation 
	$$\sum_{1 \leq m < N} (u_m - v_m)(b^m + b^{2N-m}) = 0$$ yields the solution $(u_1', \ldots, u_{N-1}', 0) \times (v_1', \ldots, v_{N-1}', 0)$ to the equation $$\sum_{1 \leq m \leq N}(u_m-v_m)(b^m + b^{2N - m})=0.$$ Since $r(0;K,b)=1$ additionally, the claim follows.

	For any integer $n$, Lemma \ref{lem:comp2} implies
	$$
	r(n;K,b) \leq  \dfrac{b^K\mathbf{1}_{0 \leq n \leq(b-1)K}}{\sqrt{(b^2-1)K}} \left(\sqrt{\dfrac{6}{\pi}}\exp\left(-\dfrac{6}{(b^2-1)K}\left(n - \dfrac{(b-1)K}{2}\right)^2 \right) + \dfrac{c}{K}\right) 
	$$
	for some absolute constant $c > 0$. We fix a smooth function $\nu : \mathbb{R} \to [0,1]$ compactly supported on $[-2,2]$ satisfying $\nu(t) = 1$ for $0 \leq t \leq 1$ and $\|\nu^{(j)}\|_\infty \ll_j 1$ for each $j\geq 0$. In particular $\mathbf{1}_{0 \leq n \leq (b-1)K} \leq \nu(n/(b-1)K)$ and
	$$
	r(n;K,b) \leq  \dfrac{b^K  \nu(\frac{n}{(b-1)K})}{\sqrt{(b^2-1)K}} \left(\sqrt{\dfrac{6}{\pi}}\exp\left(-\dfrac{6}{(b^2-1)K}\left(n - \dfrac{(b-1)K}{2}\right)^2 \right) + \dfrac{c}{K}\right).
	$$ 
	We can rewrite this as
	$$
	r(n;K,b) \leq  \dfrac{b^K}{\sqrt{(b^2-1)K}} \eta_{K,b} \left(\dfrac{n}{\sqrt{(b^2-1)K}}\right),
	$$
	where $\eta_{K,b} : \mathbb{R} \to \mathbb{R}^+_0$ is the smooth compactly supported function defined by 
	\begin{equation}\label{eqn: def2 of eta}
	\eta_{K,b}(t) = \left(\sqrt{\dfrac{6}{\pi}} \exp\left(-6\left(t - \dfrac{1}{2} \sqrt{\dfrac{(b-1)K}{b+1}}\right)^2\right) + \dfrac{c}{K} \right)\nu\left(t \sqrt{\dfrac{b+1}{(b-1)K}}\right).
	\end{equation}
	Then we have 
	$$
	\sum_{\substack{0\leq u_1, \ldots, u_{N} \leq (b-1)K \\ 0 \leq v_1, \ldots, v_{N} \leq (b-1)K \\ \sum_{1 \leq m \leq N}(u_m - v_m)(b^m + b^{2N-m}) = 0}} \prod_{1 \leq n \leq N} r(u_n; K,b)r(v_n;K,b) \leq \dfrac{b^{2KN}}{(b^2-1)^NK^N}J,
	$$
	where $J$ is the sum
	$$
	\sum_{\substack{ u_1, \ldots, u_{N} \in \mathbb{Z} \\ v_1, \ldots, v_{N} \in \mathbb{Z} \\ \sum_{1 \leq m \leq N}(u_m - v_m)(b^m + b^{2N-m}) = 0}} \prod_{1 \leq n \leq N} \eta_{K,b}\left(\dfrac{u_n}{\sqrt{(b^2-1)K}}\right) \eta_{K,b}\left(\dfrac{v_n}{\sqrt{(b^2-1)K}}\right).
	$$
	In view of (\ref{eqn: bound for sum of squares}) we complete the proof if we show that 
	\begin{equation}\label{eqn: desired bound}
	J \leq b^2 \dfrac{(b^2-1)^N K^N}{b^{2N}} \left(1 + O\left(\dfrac{1}{\sqrt{K}} + \dfrac{b^2}{K}\right)\right)^{2N}. 
	\end{equation}
	To this end we proceed as follows.
	
	An application of the identity 
	$$
	\mathbf{1}_{k=0} = \int_0^1 e(k\alpha) d\alpha
	$$
	valid for integers $k$, gives
	$$
	J = \int_{0}^1 \left|\prod_{1 \leq n \leq N}\sum_{v \in \mathbb{Z}} \eta_{K,b}\left(\dfrac{v}{\sqrt{(b^2-1)K}}\right)e\left(\alpha v \left(b^n + b^{2N-n}\right)\right)\right|^2d\alpha.
	$$
	Let us set
	\begin{equation}\label{def: A}
	A := \sqrt{(b^2-1)K}\|\eta_{K,b}\|_1 + \|\eta_{K,b}'\|_1.
	\end{equation}
	By the symmetry of $n \mapsto b^n + b^{2N-n}$ around $n = N$ and the bound 
	$$
	\left|\sum_{v \in \mathbb{Z}} \eta_{K,b}\left( \dfrac{v}{\sqrt{(b^2-1)K}} \right) e\left(2\alpha v b^N \right) \right| \leq A
	$$
	coming from (\ref{eqn: trivial bound}) in Lemma \ref{lem: smooth exp sums}, note
	$$
	J \leq A \int_0^1 \prod_{1 \leq n < 2N}\left|\sum_{v \in \mathbb{Z}}\eta_{K,b}\left(\dfrac{v}{\sqrt{(b^2-1)K}}\right)e\left(\alpha v\left(b^n + b^{2N-n}\right)\right)\right|d\alpha.
	$$
	Grouping by pairs of adjacent factors and bounding square roots of two endpoint factors, we have 
	\begin{align*}
	&\prod_{1 \leq n < 2N}\left|\sum_{v \in \mathbb{Z}}\eta_{K,b}\left(\dfrac{v}{\sqrt{(b^2-1)K}}\right)e\left(\alpha v\left(b^n + b^{2N-n}\right)\right)\right|\\
	&\leq
	A
	\prod_{1 \leq n \leq 2(N-1)}
	\sqrt{\left|
	\sum_{u \in \mathbb{Z}}\eta_{K,b}\left(\dfrac{u}{\sqrt{(b^2-1)K}}\right)e\left(\alpha u\left(b^n + b^{2N-n}\right)\right)\right|}\\
	&\qquad
	\times
	\sqrt{\left|\sum_{v \in \mathbb{Z}}\eta_{K,b}\left(\dfrac{v}{\sqrt{(b^2-1)K}}\right)e\left(\alpha v\left(b^{n+1} + b^{2N-n-1}\right)\right)
	\right|}.
	\end{align*}
	By Lemma \ref{lem: smooth exp sums} with $k=4$, the product over $1 \leq n \leq 2(N-1)$ is at most
	\begin{align*}
	&\prod_{1 \leq n \leq 2(N-1)} \sqrt{\min\left(A, \dfrac{\|\eta_{K,b}^{(4)}\|_1}{(b^2-1)^{3/2}K^{3/2}\sin^4(\pi \alpha(b^n + b^{2N-n}))}\right)}\\
	&\qquad 
	\times
	\sqrt{\min\left(A, \dfrac{\|\eta_{K,b}^{(4)}\|_1}{(b^2-1)^{3/2}K^{3/2}\sin^4(\pi \alpha(b^{n+1} + b^{2N-n-1}))}\right)}.
	\end{align*}
	By Lemma \ref{lem: bound for product of pair} with $\beta = 2N - n - 1$ and $\gamma = n$, the fact that $t \mapsto \sin(\pi t)$ is increasing on $[0,1/2]$ and $|\sin(\pi t)| = \sin(\pi \|t\|)$, we have either
	$$
	\left|\sin\left(\pi \alpha \left(b^n + b^{2N-n}\right)\right)\right| \geq \sin\left(\pi \dfrac{\|\alpha(b^2-1)b^n\|}{b+1}\right)
	$$
	or
	$$
	\left|\sin\left(\pi \alpha \left(b^{n+1} + b^{2N-n-1}\right)\right)\right| \geq \sin\left(\pi \dfrac{\|\alpha(b^2-1)b^n\|}{b+1}\right).
	$$
	Then the product above is
	\begin{align*}
	&\leq A^{N-1}
	\prod_{1 \leq n \leq 2(N-1)} \sqrt{\min\left(A, \dfrac{\|\eta_{K,b}^{(4)}\|_1}{(b^2-1)^{3/2}K^{3/2}\sin^4(\pi\|\alpha(b^2-1)b^n\|/(b+1))}\right)}\\
	&\leq
	A^{N-1}
	\prod_{1 \leq n \leq 2(N-1)} \sqrt{\min\left(A, \dfrac{(b+1)^4\|\eta_{K,b}^{(4)}\|_1}{(b^2-1)^{3/2}K^{3/2}\sin^4(\pi \alpha(b^2-1)b^n)}\right)}.
	\end{align*}
	Note in the last line we used the inequality $r \sin(\pi t/r) \geq \sin(\pi t)$ valid for $r \geq 1$ and $0 \leq t < 1$ (which may be shown to hold, say, via Euler's product formula). Thus
	\begin{align*}
	&\prod_{1 \leq n < 2N}\left|\sum_{v \in \mathbb{Z}}\eta_{K,b}\left(\dfrac{v}{\sqrt{(b^2-1)K}}\right)e\left(\alpha v\left(b^n + b^{2N-n}\right)\right)\right|\\
	&
	\leq
	A^{N}
	\prod_{1 \leq n \leq 2(N-1)} \sqrt{\min\left(A, \dfrac{(b+1)^4\|\eta_{K,b}^{(4)}\|_1}{(b^2-1)^{3/2}K^{3/2}\sin^4(\pi \alpha(b^2-1)b^n)}\right)}.
	\end{align*}
	Inserting this in the integral, substituting $\alpha b(b^2-1)$ with $\alpha$ and using the $1$-periodicity of the integrand, we obtain
	\begin{align}
	J &\leq A^{N+1}
	\int_0^1 \prod_{0 \leq n < 2(N-1)}\sqrt{\min\left(A, \dfrac{(b+1)^4\|\eta_{K,b}^{(4)}\|_1}{(b^2-1)^{3/2}K^{3/2}\sin^4(\pi \alpha b^n)}\right)}d\alpha. \label{eqn4}
	\end{align}
	By Lemma \ref{lem: integral of product} and the fact $\sqrt{\min(A,B)} = \min(\sqrt{A}, \sqrt{B})$ for $A,B \geq 0$, this integral (without the normalizing factor $A^{N+1}$) is
	$$
	\leq \left(\sup_{\theta \in \mathbb{R}} \dfrac{1}{b} \sum_{n (b)} \min\left(A^{1/2},
	\dfrac{(b+1)^2\|\eta_{K,b}^{(4)}\|_1^{1/2}}{(b^2-1)^{3/4}K^{3/4}\sin^2(\pi (n+\theta)/b)}
	\right)\right)^{2(N-1)}
	$$  
	while Lemma \ref{lem: vinogradov} implies this is
	$$
	\leq \left(\dfrac{A^{1/2}}{b} + \dfrac{b(b+1)^2\|\eta_{K,b}^{(4)}\|_1^{1/2}}{(b^2-1)^{3/4}K^{3/4}} \right)^{2(N-1)}.
	$$
	Factoring $\frac{1}{b}(b^2-1)^{1/4}K^{1/4}$ out from the inside of the brackets and factoring\\ $(b^2-1)^\frac{N+1}{2}K^{\frac{N+1}{2}}$ out from the normalizing factor $A^{N+1}$ in (\ref{eqn4}) it follows
	$$
	J \leq b^2\dfrac{(b^2-1)^N K^N}{b^{2N}}\mathcal{E}(N,K,b),
	$$
	where
	\begin{align*}
	&\mathcal{E}(N,K,b)\\
	&:= \left(\dfrac{A}{\sqrt{(b^2-1)K}}\right)^{N+1}\left(\dfrac{A^{1/2}}{(b^2-1)^{1/4}K^{1/4}} + 
	\dfrac{b^2(b+1)^2\|\eta_{K,b}^{(4)}\|_1^{1/2}}{(b^2-1)K}
	\right)^{2(N-1)}\\
	&\leq
	\left(\dfrac{A}{\sqrt{(b^2-1)K}}\right)^{N}\left(\dfrac{A^{1/2}}{(b^2-1)^{1/4}K^{1/4}} + 
	\dfrac{3b^2\|\eta_{K,b}^{(4)}\|_1^{1/2}}{K}
	\right)^{2N}\\
	&=
	\left(\|\eta_{K,b}\|_1 + \dfrac{\|\eta_{K,b}'\|_1}{\sqrt{(b^2-1)K}}
	+
	\dfrac{3b^2\|\eta_{K,b}^{(4)}\|_1^{1/2}}{K}\sqrt{\|\eta_{K,b}\|_1 + \dfrac{\|\eta_{K,b}'\|_1}{\sqrt{(b^2-1)K}}}
	\right)^{2N}.
	\end{align*}
	The inequality above follows from the facts $B  (\sqrt{B} + C)^{-2} \leq 1$ for any $B> 0$, $C \geq 0$, and $(b+1)^2 \leq 3(b^2-1)$ for $b\geq 2$. 
	The last equality holds by the definition of $A$ in (\ref{def: A}). From the definition of $\eta_{K,b}$ in (\ref{eqn: def2 of eta}) and the assumptions on $\nu$, we have $\|\eta_{K,b}'\|_1, \|\eta_{K,b}^{(4)}\|_1 \ll 1$ uniformly in $K,b$. Moreover the fact $\int_{\mathbb{R}}e^{-t^2}dt = \sqrt{\pi}$ and the assumption $\nu \leq 1$ imply
	$$
	\|\eta_{K,b}\|_1 \leq 1 +  \dfrac{c\|\nu\|_1}{\sqrt{K}} = 1 + O\left(\dfrac{1}{\sqrt{K}}\right). 
	$$
	Hence
	$$
	\mathcal{E}(N,K,b) \leq \left(1 + O\left(\dfrac{1}{\sqrt{K}} + \dfrac{b^2}{K}\right)\right)^{2N}
	$$
	and
	$$
	J \leq b^2 \dfrac{(b^2-1)^N K^N}{b^{2N}} \left(1 + O\left(\dfrac{1}{\sqrt{K}} + \dfrac{b^2}{K}\right)\right)^{2N}.
	$$
	We have shown that (\ref{eqn: desired bound}) holds and the proof is thus complete.
\end{proof}	

\section{Bounding the average}\label{sec: average bound}

Here we prove Proposition \ref{prop: average bound} below combining our moment bounds in Proposition \ref{prop:2K-th moment2} with Col's $L^\infty$-type bound in Proposition \ref{prop: Linfty bound}. As an application of Proposition \ref{prop: average bound}, we then give a proof of Theorem \ref{thm: equidistribution} in the following section. 

\begin{prop}
	[$L^1/L^2/L^{2K}/L^\infty$ hybrid bound]\label{prop: average bound}
	Let $N\geq 0$ be an integer and let $Q\geq 1$. For any $0 < \epsilon \leq \frac{1}{15}$ and $\frac{1}{3} \leq \delta \leq \frac{2}{5} - \epsilon$, 
	\begin{equation}\label{eqn: general L1 bound}
		\sup_{\beta \in \mathbb{R}}\sum_{\substack{q \leq Q \\ (q,b)=1}} \sum_{h(q)}^* \Phi_{N}\left(\dfrac{h}{q} + \beta\right) \ll_{b,\epsilon} Q^2b^{N(1 - \delta - \sigma_1(b,\epsilon))} + Q^{1 - \frac{\sigma_1(b,\epsilon)}{\delta}}b^N,
	\end{equation}
	where $\sigma_1(b,\epsilon) > 0$ is some value depending only on $b$ and $\epsilon$. 
	Moreover 
	\begin{equation}\label{eqn: L1 bound for r dividing b(b^2-1)}
		\sum_{\substack{2 \leq q \leq Q \\ (q,b^3-b)=1}} \sum_{h(q)}^* \Phi_{N}\left(\dfrac{h}{q} + \dfrac{k}{b^3-b}\right) \ll_{b,\epsilon} \left(Q^2b^{N(1 - \delta - \sigma_1(b,\epsilon))} + Q^{1 - \frac{\sigma_1(b,\epsilon)}{\delta}}b^N\right)e^{-\frac{\sigma_\infty(b)N}{\log Q}}
	\end{equation}
uniformly in $k \in \mathbb{Z}$, 
	where $\sigma_\infty(b) > 0$ is some value depending only on $b$.
\end{prop}

We will need the following bound. 

\begin{lem}[$L^2$-bound]\label{lem: L2 bound}
	Let $0 \leq L < M < N$ be integers and let $Q \geq 1$. Then
	$$
	\sup_{\beta \in \mathbb{R}}\sum_{q\leq Q} \sum_{h(q)} \prod_{L < n \leq M} \phi_b^2\left(\left(\dfrac{h}{q} + \beta\right)\left(b^n + b^{2N-n}\right)\right) \ll_\epsilon \left(Q + b^{M - L + \epsilon N}\right)Qb^{M-L}
	$$	
	for any $\epsilon > 0$.
\end{lem}	

\begin{proof}
	Expanding the square, switching orders of summation, using the orthogonality of the additive characters modulo $q$ and taking absolute values, we have that the left hand side above is
	$$
	\leq Q\sum_{\substack{0 \leq u_{L+1}, \ldots, u_{M} < b \\ 0\leq v_{L+1}, \ldots, v_{M} < b}} \sum_{\substack{q \leq Q \\ q \mid S(\mathbf{u}, \mathbf{v})}}1,
	$$
	where
	$$
	S(\mathbf{u}, \mathbf{v}) := \sum_{L < n \leq M} \left(u_n - v_n\right)\left(b^n + b^{2N-n}\right). 
	$$
	By the uniqueness of the representation of integers in base $b$, we have $S(\mathbf{u}, \mathbf{v}) =0$ if and only if $u_n = v_n$ for each $L < n\leq M$. The contribution to the overall sum of such (diagonal) terms is $\leq Q^2b^{M-L}$. For the off-diagonal terms with $u_n \neq v_n$ for some $n$, the inner sum is $\leq \tau(S(\mathbf{u}, \mathbf{v})) \ll_\epsilon b^{\epsilon N}$ for any $\epsilon > 0$, where $\tau(m) := \sum_{d \mid m }1$ is the divisor function. Thus these contribute $\ll_\epsilon Q b^{2(M-L) + \epsilon N}$ for any $\epsilon > 0$. 
\end{proof}

\begin{proof}[{\bf Proof of Proposition \ref{prop: average bound}}]
	We may assume that $N\geq 100$, say, as otherwise the statement is trivial. 
	Let us treat the cases when $Q \ll_{b,\epsilon} b^{\delta N}$ and $Q \gg_{b,\epsilon} b^{\delta N}$ separately. 
	
	We begin by considering the case when $Q \gg_{b,\epsilon} b^{\delta N}$. Here one can check that both right hand sides of (\ref{eqn: general L1 bound}) and (\ref{eqn: L1 bound for r dividing b(b^2-1)}) are $\gg_{b,\epsilon} Q^2b^{N(1-\delta - \sigma_1(b,\epsilon))}$ (in particular the exponential factor in (\ref{eqn: L1 bound for r dividing b(b^2-1)}) is $\asymp_{b,\epsilon} 1$ since $\delta \geq 1/3$ by assumption). Thus to prove (\ref{eqn: general L1 bound}) and (\ref{eqn: L1 bound for r dividing b(b^2-1)}) when $Q \gg_{b,\epsilon} b^{\delta N}$ it suffices to show that the left hand side of (\ref{eqn: general L1 bound}) is $\ll_{b,\epsilon} Q^2b^{N(1-\delta - \sigma_1(b,\epsilon))}$ in this case. To this end we proceed as follows.
	
	Let $M = \lfloor \delta N \rfloor$ and split the product as  
	$$
	\Phi_N\left(\dfrac{h}{q} + \beta\right) = P_1\left(\dfrac{h}{q}\right)P_2\left(\dfrac{h}{q}\right)P_3\left(\dfrac{h}{q}\right),
	$$	
	where
	\begin{align*}
		P_1\left(\dfrac{h}{q}\right) 
		&=
		\prod_{1 \leq n \leq M} \phi_b\left(\left(\dfrac{h}{q} + \beta\right)\left(b^n + b^{2N-n}\right)\right)
		,\\
		P_2\left(\dfrac{h}{q}\right) 
		&=
		\prod_{M < n \leq 2M}\phi_b\left(\left(\dfrac{h}{q} + \beta\right)\left(b^n + b^{2N-n}\right)\right)
		,\\
		P_3\left(\dfrac{h}{q}\right) 
		&=
		\prod_{2M < n < N}
		\phi_b\left(\left(\dfrac{h}{q} + \beta\right)\left(b^n + b^{2N-n}\right)\right)
		.
	\end{align*}
	Let $K \geq 2$ be an integer to be specified later and 
	set $\ell := 4K/(2K-1) > 2$. Note that $2/\ell + 1/2K= 1$. Then by  H\"{o}lder's inequality with the triple $(\ell,\ell,2K)$, we have
	\begin{align*}
	&\sum_{\substack{q \leq Q \\ (q,b)=1}}\sum_{h(q)}^* \Phi_N\left(\dfrac{h}{q} + \beta\right)\\ 
		& \leq 
		\left(\sum_{\substack{q \leq Q }}\sum_{h(q)}P_1^\ell\left(\dfrac{h}{q}\right) \right)^{1/\ell}\left(\sum_{\substack{q \leq Q }}\sum_{h(q)}P_2^\ell\left(\dfrac{h}{q}\right)\right)^{1/\ell}\left(\sum_{\substack{q \leq Q \\ (q,b)=1}}\sum_{h(q)}^* P_3^{2K}\left(\dfrac{h}{q}\right) \right)^{1/2K}.
	\end{align*}
	Note $P_1^\ell(h/q) = P_1^{\ell-2}(h/q) P_1^{2}(h/q) \leq b^{(\ell-2)M} P_1^{2}(h/q)$ while Lemma \ref{lem: L2 bound} gives
	$$
	\sum_{q \leq Q}\sum_{h(q)} P_1^2\left(\dfrac{h}{q}\right) \ll_\gamma \left(Q + b^{M  }\right)Qb^{M + \gamma N}
	$$
	for any $\gamma > 0$.
	Thus
	$$
	\left(\sum_{\substack{q \leq Q }}\sum_{h(q)}P_1^\ell\left(\dfrac{h}{q}\right) \right)^{1/\ell} \ll_\gamma \left(Q + b^{M }\right)^{1/\ell}Q^{1/\ell}b^{M(1 - 1/\ell) + \gamma N}
	$$
	for any $\gamma > 0$. 
	Similarly we can show that 
	$$
	\left(\sum_{\substack{q \leq Q }}\sum_{h(q)}P_2^\ell\left(\dfrac{h}{q}\right) \right)^{1/\ell} \ll_\gamma \left(Q + b^{M }\right)^{1/\ell}Q^{1/\ell}b^{M(1 - 1/\ell) + \gamma N}.
	$$
	
	Consider now the third sum with $P_3$. We have
	\begin{align*}
		\sum_{h(q)}^*P_3^{2K}\left(\dfrac{h}{q}\right) &= \sum_{h(q)}^* \prod_{1 \leq n < N-2M} \phi_b^{2K} \left( \left(\dfrac{h}{q} + \beta \right)\left(b^{2M + n} + b^{2N - 2M - n}\right)\right)\\
		&=
		\sum_{h(q)}^* \prod_{1 \leq n < N-2M} \phi_b^{2K} \left( \left(\dfrac{hb^{2M}}{q} + \beta b^{2M}\right) \left(b^{ n} + b^{2(N - 2M) - n}\right)\right)\\
		&= 
		\sum_{h(q)}^*\Phi_{N-2M}^{2K}\left(\dfrac{hb^{2M}}{q} + \beta b^{2M}\right) .
	\end{align*}
	If $(q,b)=1$, we may substitute $hb^{2M}$ with $h$ above. Thus
	$$
	\sum_{\substack{q \leq Q \\ (q,b)=1}}\sum_{h(q)}^*P_3^{2K}\left(\dfrac{h}{q}\right)
	= 
	\sum_{\substack{q \leq Q \\ (q,b)=1}}\sum_{h(q)}^*\Phi_{N-2M}^{2K}\left(\dfrac{h}{q} + \beta b^{2M}\right).
	$$	
	Removing now the constraint $(q,b)=1$ via positivity, Proposition \ref{prop:2K-th moment2} implies the above is
	$$
	\leq \left(Q + \sqrt{K}b^{N - 2M}\right)^2b^{2(K-1)(N - 2M) + 2} \left(1 + \dfrac{c}{\sqrt{K}} + \dfrac{cb^2}{K}\right)^{2(N-2M)}
	$$
	for some absolute constant $c > 0$.
	Multiplying the three bounds together and recalling the assumptions $M = \lfloor \delta N \rfloor$ and $2/\ell + 1/2K=1$ gives
	\begin{align}
		&\sum_{\substack{q \leq Q \\ (q,b)=1}}\sum_{h(q)}^* \Phi_N\left(\dfrac{h}{q} + \beta\right)\nonumber\\
		& 
		\ll_{b, \gamma} \left(Q + b^{M }\right)^{\frac{2}{\ell}}\left(Q + b^{N- 2M}\right)^{\frac{1}{K}}Q^{\frac{2}{\ell}}b^{N - \frac{N}{K} - \frac{2M}{\ell} + \frac{2M}{K} + \gamma N}\left(1 + \dfrac{cb^2}{K} + \dfrac{c}{\sqrt{K}}\right)^{\frac{N-2M}{K}}\nonumber\\
		&\ll_b \left(Q + b^{\delta N}\right)^{1-1/2K}\left(Q + b^{N(1-2\delta)}\right)^{1/K}Q^{1-1/2K} b^{N(1 - \delta - \frac{1}{K}L(\delta, b,K,\gamma))}\label{eqn: last line},
	\end{align}
	where
	\begin{align*}
		L(\delta,b,K,\gamma) &:= 1-\dfrac{5\delta}{2} - \gamma K - (1-2\delta)\log_b\left(1 + \dfrac{cb^2}{K} + \dfrac{c}{\sqrt{K}}\right)\\
		&\geq \dfrac{5\epsilon}{2}
		- \gamma K - \dfrac{1}{3}\log_b\left(1 + \dfrac{cb^2}{K} + \dfrac{c}{\sqrt{K}}\right).
	\end{align*}
	The last line holds by the assumption $\frac{1}{3} \leq \delta \leq \frac{2}{5} - \epsilon$. We may choose $K \asymp_{b,\epsilon} 1$ large enough and $\gamma \ll_{b,\epsilon} 1$ small enough so that 
	$$
	\gamma K + \dfrac{1}{3}\log_b\left(1 + \dfrac{cb^2}{K} + \dfrac{c}{\sqrt{K}}\right) \leq \dfrac{3\epsilon}{2},
	$$
	say. For any such choice, $L(\delta,b,K,\gamma) \geq \epsilon$. This and the  assumption $Q \gg_{b,\epsilon} b^{\delta N} \geq b^{(1-2\delta)N}$ for $\delta \geq 1/3$ implies (\ref{eqn: last line}) is
	$
	\ll_{b,\epsilon} Q^2 b^{N(1 - \delta - \epsilon/K)}.
	$
	Thus if we take $\sigma_1(b,\epsilon) = \epsilon/K$ we see that (\ref{eqn: general L1 bound}) and hence (\ref{eqn: L1 bound for r dividing b(b^2-1)}) hold in our case of $Q \gg_{b,\epsilon} b^{\delta N}$.
	
	Next we treat the case when $Q\ll_{b,\epsilon} b^{\delta N}$. Here one observes that the right hand sides of (\ref{eqn: general L1 bound}) and (\ref{eqn: L1 bound for r dividing b(b^2-1)}) are $\gg_{b,\epsilon} b^N Q^{1 - \frac{\sigma_1(b,\epsilon)}{\delta}}$ and $\gg_{b,\epsilon} b^N Q^{1 - \frac{\sigma_1(b,\epsilon)}{\delta}}e^{-\sigma_\infty(b)N/\log Q}$, respectively. It thus suffices to show that the left hand sides of (\ref{eqn: general L1 bound}) and (\ref{eqn: L1 bound for r dividing b(b^2-1)}) are $\ll_{b,\epsilon} b^N Q^{1 - \frac{\sigma_1(b,\epsilon)}{\delta}}$ and $\ll_{b,\epsilon} b^N Q^{1 - \frac{\sigma_1(b,\epsilon)}{\delta}}e^{-\sigma_\infty(b)N/\log Q}$, respectively. We treat these two in similar fashion (but with a small difference) as follows.
	
	Since $Q \ll_{b,\epsilon} b^{\delta N}$ by assumption, we can find an integer $1 \leq M < N-1$ such that $b^{\delta(N-M)} \asymp_{b,\epsilon} Q$. For any such choice of $M$, we split the products according to whether $n \leq M$ or $M < n< N$. This yields that the left hand sides of (\ref{eqn: general L1 bound})
	and (\ref{eqn: L1 bound for r dividing b(b^2-1)}) are
	$
	\leq b^{M}
	\mathscr{Z}
	$ (with $\mathscr{Z}$ defined below in (\ref{def: definition of Z})) and
	$$
	\leq \mathscr{Z}\max_{\substack{h \in \mathbb{Z} \\ 2 \leq q \leq Q \\ (q,h(b^3-b))=1}}\prod_{1 \leq n \leq M} \phi_b\left(\left(\dfrac{h}{q} + \dfrac{k}{b^3-b}\right)\left(b^n + b^{2N-n}\right)\right)
	$$
	respectively, 
	where
	\begin{equation}\label{def: definition of Z}
	\mathscr{Z} := \sup_{\beta \in \mathbb{R}}\sum_{\substack{q \leq Q \\ (q,b)=1}} \sum_{h(q)}^* \prod_{M < n < N} \phi_b\left(\left(\dfrac{h}{q} + \beta \right)\left(b^n + b^{2N-n}\right)\right).
	\end{equation}
	With regards to the product over $1 \leq n \leq M$ above, Proposition \ref{prop: Linfty bound} implies it is
	$$
	\ll_b b^M \exp\left(-\sigma_\infty(b) \dfrac{M}{\log Q}\right) \ll_{b,\epsilon} b^M \exp\left(-\sigma_\infty(b) \dfrac{N}{\log Q}\right).
	$$ 
	Note we used the assumption $b^{\delta(N-M)} \asymp_{b,\epsilon} Q$ implying $M = N - \frac{1}{\delta} \log_b Q + O_{b,\epsilon}(1)$. 
	Similarly as done previously, one can show that 
	$$
	\sum_{\substack{q \leq Q \\ (q,b)=1}} \sum_{h(q)}^* \prod_{M < n < N} \phi_b\left(\left(\dfrac{h}{q} + \beta \right)\left(b^n + b^{2N-n}\right)\right) = \sum_{\substack{q \leq Q \\ (q,b)=1}} \sum_{h(q)}^* \Phi_{N-M}\left(\dfrac{h}{q} + b^M \beta \right).
	$$
	Recall that $ b^{\delta(N-M)}\asymp_{b,\epsilon} Q $ by assumption. Since we already showed above that (\ref{eqn: general L1 bound}) holds for arbitrary $N \geq 1$ when $Q \gg_{b,\epsilon } b^{\delta N}$ (and in particular when $Q \asymp_{b,\epsilon} b^{\delta N}$) we may apply this to our case of $Q \asymp_{b,\epsilon} b^{\delta(N-M)}$ (substituting $N$ there with $N-M$ here). This shows 
	$$
	\mathscr{Z} \ll_{b,\epsilon} Q^2 b^{( N -M)(1 - \delta - \sigma_1(b,\epsilon))} + Q^{1-\frac{\sigma_1(b,\epsilon)}{\delta}} b^{N-M} \asymp_{b,\epsilon} Q^{1-\frac{\sigma_1(b,\epsilon)}{\delta}} b^{N-M}.
	$$
	It follows that (\ref{eqn: general L1 bound}) and (\ref{eqn: L1 bound for r dividing b(b^2-1)}) also hold when $Q \ll_{b,\epsilon} b^{\delta N}$ and we thus conclude the proof. 
\end{proof}	

\section{Equidistribution estimate}\label{sec: equi estimate}

We now prove Theorem \ref{thm: equidistribution} as an application of Proposition \ref{prop: average bound}. 
First we need the following fact. 

\begin{lem}\label{lem: number of coprimes}
	For an integer $N\geq 0$, let $\Pi_b(2N)$ be as defined in (\ref{eqn: def of Pi}) and let 
	$$
	\Pi_b^*(2N) := \left\{ n \in \Pi_b(2N) \ : \ (n,b^3-b)=1\right\}. 
	$$
	We have $\#\Pi_b(2N) = (b-1)b^N$ and 
	\begin{equation}\label{bound: size of Pi star}
	\#\Pi_b^*(2N) = \gamma_2(b) \dfrac{\varphi(b^3-b)}{b^3-b} \#\Pi_b(2N) + O\left(b^2 \tau(b^2-1)\right), 
	\end{equation}
	where
	$$
	\gamma_2(b) := \begin{cases}
		b/(b-1) &\mbox{ if $b$ is even,} \\
		1 &\mbox{ otherwise.}
	\end{cases}
	$$
	Moreover $\#\mathscr{P}_b^*(x) \asymp_b \sqrt{x}$ for $x \geq 1$. 
\end{lem}

\begin{proof}
	Every integer $n \in \Pi_b(2N)$ can be written uniquely as 
	\begin{align*}
		n &= n_N b^N + \sum_{0 \leq j < N} n_j \left(b^j + b^{2N-j}\right)\\
		&=
		n_N b^N + 2\sum_{0 \leq j < N} n_j b^j
		+ \sum_{0 \leq j < N} n_j \left( b^{2N-j} - b^j\right)
	\end{align*}	
	for some unique digits $0 \leq n_j < b$ with $n_0 > 0$. Thus $\#\Pi_b(2N) = (b-1)b^{N}$. We note that every term of the right-most sum above is divisible by $b^2-1$. Then 
	$$(n,b^2-1) = \left(n_N b^N + 2\sum_{0 \leq j < N} n_j b^j, \ b^2-1\right).$$ 
	Since $(n,b)=(n_0,b)$ and $b^3-b = b(b^2-1)$, it follows
	\begin{align*}
		\#\Pi_b^*(2N) &= \sum_{\substack{0 \leq n_0, \ldots, n_N < b \\ (n_0,b) = (n_Nb^N + 2\sum_{0 \leq j < N} n_j b^j, b^2-1)=1}} 1\\ 
		&= \sum_{\substack{0 \leq n_0 < b \\ (n_0,b)=1}} \sum_{0 \leq n_N < b} \sum_{\substack{0 \leq n < b^{N-1} \\ (n_N b^N + 2n_0 + 2bn, b^2-1) = 1}}1 \\ 
		&= 
		\sum_{\substack{0 \leq n_0 < b \\ (n_0,b)=1}} \sum_{\substack{0 \leq n_N < b \\  (n_N,2)=1 \text{ if } (b,2)=1}} \sum_{\substack{0 \leq n < b^{N-1} \\ ( n_N b^N + 2n_0 +2bn, b^2-1) = 1}}1.
	\end{align*}
	The equality before the last follows from the unique representation of integers in base $b$.
	The last equality follows from the fact that if $b$ is odd, whence $b^2-1$ is even, then $$\left(n_N b^N + 2n_0 + 2bn, \ b^2-1\right) = 1$$ implies $n_N$ is odd. 
	
	The inner sum equals $S(n_Nb^N + 2n_0)$, where for an integer $a$, 
	$$
	S(a) := \sum_{\substack{0 \leq n < b^{N-1} \\ (a + 2bn, b^2-1) = 1}}1.
	$$
	If $b$ odd and $a$ is even, $S(a) = 0$. Thus we assume $a$ is odd if so is $b$. By the M\"{o}bius inversion formula,
	$$
	S(a) = \sum_{d \mid (b^2-1)} \mu(d)  \sum_{\substack{0 \leq n < b^{N-1} \\ 2bn \equiv -a (d)}}1.
	$$ 
	If $b$ is odd and $2bn \equiv -a(d)$, then $d$ is odd as so is $a$ by assumption. In the case when $b$ is even, $b^2-1$ is odd and so is any divisor $d$ of $b^2-1$. Thus $d$ runs over odd divisors of $b^2-1$ regardless. These are also coprime to $b$ as $(b^2-1,b)=1$. Letting $\overline{2b}$ denote the inverse of $2b$ modulo $d$, the above is 
	\begin{align*}
		S(a) &= \sum_{\substack{d \mid (b^2-1) \\ (d,2)=1}} \mu(d)  \sum_{\substack{0 \leq n < b^{N-1} \\ n \equiv -\overline{2b}a (d)}}1
		= \sum_{\substack{d \mid (b^2-1) \\ (d,2)=1}} \mu(d) \left(\dfrac{b^{N-1}}{d} + O(1)\right)\\
		&= \gamma(b) \dfrac{\varphi(b^2-1)}{b^2-1} b^{N-1}  + O(\tau(b^2-1)) 
	\end{align*}
	uniformly in $a$,
	where
	$$
	\gamma(b) := \begin{cases}
		1 &\mbox{ if $b$ is even}\\
		2 &\mbox{ if $b$ is odd}.
	\end{cases}
	$$ 
	Now (\ref{bound: size of Pi star}) follows after we insert this into the expression for $\#\Pi_b^*(2N)$ above, compute the resulting sum and recall that $\#\Pi_b(2N) = (b-1)b^N$.

	We now establish $\sqrt{x} \ll_b \#\mathscr{P}_b^*(x) \ll_b \sqrt{x}$ for $x \geq 1$. Since $\mathscr{P}_b^*(x) \subseteq \mathscr{P}_b(x) $ and $\#\mathscr{P}_b(x) \asymp_b \sqrt{x}$, the upper bound holds. To show $\#\mathscr{P}_b^*(x) \gg_b \sqrt{x}$, we first note that since $1 \in \mathscr{P}_b^*(x)$, the lower bound holds for $\log_b x$ bounded and we may assume then that $\log_b x 
	$ is arbitrarily large. In this case, let $N$ be the largest integer such that $\Pi_b(2N) \subseteq \mathscr{P}_b(x)$. 
	Clearly $b^N \asymp_b \sqrt{x}$ and $\#\mathscr{P}_b^*(x) \geq \#\Pi_b^*(2N) \asymp_b b^N$ for $N$ large.   
\end{proof}

\begin{proof}[{\bf Proof of Theorem \ref{thm: equidistribution}}]
	Let $y \leq x$. 
	Since every $b$-palindromic integer $n\geq 1$ with $\lfloor \log_b n\rfloor$ odd is divisible by $b+1$, and $b+1 > 1$ is a divisor of $ b^3-b$, it follows
	$$
	\mathscr{P}^*_b(y) = \left\{n \in \mathscr{P}_b^0(y) \ : \ (n,b^3-b)=1\right\}. 
	$$
	Then by this and the M\"{o}bius inversion formula $\mathbf{1}_{(b^3-b,n)=1} = \sum_{r \mid (b^3-b,n)}\mu(r)$, we have
	$$
	\#\mathscr{P}_b^*(y,a,q) - \dfrac{\#\mathscr{P}_b^*(y)}{q} =\sum_{r \mid (b^3-b)} \mu(r) \sum_{\substack{n \in \mathscr{P}_b^0(y) \\ r\mid n}} \left(\mathbf{1}_{n \equiv a (q)} - \dfrac{1}{q}\right). 
	$$	
	Note 
	$$\mathbf{1}_{r \mid n} = \dfrac{1}{r} \sum_{k (r)} e_r(nk)$$ and 
	$$\mathbf{1}_{n \equiv a (q)} - \dfrac{1}{q} = \dfrac{1}{q} \sum_{1 \leq h < q} e_q(-ah)e_q(nh).$$
	Inserting these expressions above, switching orders of summation and taking absolute values, we obtain
	\begin{align*}
		&\left|\#\mathscr{P}_b^*(y,a,q) - \dfrac{\#\mathscr{P}_b^*(y)}{q}\right| \leq \sum_{r \mid (b^3-b)} \dfrac{1}{r}\sum_{k(r)} \dfrac{1}{q}\sum_{1 \leq h < q} \left| \sum_{n \in \mathscr{P}_b^0(y)}e\left(\left(\dfrac{h}{q} + \dfrac{k}{r}\right)n\right)\right|\\
		&\leq b^2 \sum_{0 \leq N \leq \frac{1}{2}\log_b x} \sum_{0 \leq M \leq N} \sum_{r \mid (b^3-b)} \dfrac{1}{r}\sum_{k(r)} \dfrac{1}{q}\sum_{1 \leq h < q}\Phi_M \left(\dfrac{b^{N-M}h}{q} + \dfrac{ b^{N-M}k}{r}\right)
	\end{align*}
	uniformly in $a \in \mathbb{Z}$ and $y \leq x$, by Lemma \ref{lem: bound for exp sum of pals} and a switch in the order of summation. If $(q,b)=1$, we may substitute $b^{N-M}h$ with $h$ above. Thus for any $Q\geq 1$, 
	\begin{align}\label{eqn: average discrepancy}
		&\sum_{\substack{q \leq Q \\ (q,b^3-b)=1}} \sup_{ y \leq x}\max_{a \in \mathbb{Z}} \left|\#\mathscr{P}_b^*(y,a,q) - \dfrac{\#\mathscr{P}_b^*(y)}{q}\right|\nonumber\\
		&\leq 
		b^2 \sum_{0 \leq N \leq \frac{1}{2}\log_b x} \sum_{0 \leq M \leq N} \sum_{r \mid (b^3-b)} \dfrac{1}{r}\sum_{k(r)} \sum_{\substack{q \leq Q \\ (q,b^3-b)=1}}\dfrac{1}{q}\sum_{1 \leq h < q}\Phi_M \left(\dfrac{h}{q} + \dfrac{ b^{N-M}k}{r}\right)\nonumber\\
		&\leq
		b^2 \tau(b^3-b)\sum_{0 \leq N \leq \frac{1}{2}\log_b x} \sum_{0 \leq M \leq N} \max_{\substack{r \mid (b^3-b) \\ k \in \mathbb{Z}}} \sum_{\substack{q \leq Q \\ (q,b^3-b)=1}}\dfrac{1}{q}\sum_{1 \leq h < q}\Phi_M \left(\dfrac{h}{q} + \dfrac{ k}{r}\right)\nonumber\\
		&=
		b^2 \tau(b^3-b)\sum_{0 \leq N \leq \frac{1}{2}\log_b x} \sum_{0 \leq M \leq N} \max_{\substack{ k \in \mathbb{Z}}}S(M,Q,k),
	\end{align}
	where
	$$
	S(M,Q, k) := \sum_{\substack{q \leq Q \\ (q,b^3-b)=1}} \dfrac{1}{q}\sum_{1 \leq h < q} \Phi_M\left(\dfrac{h}{q} + \dfrac{k}{b^3-b}\right). 
	$$
	Splitting the sum according to the GCD of $h,q$ and substituting variables, we have
	$$
	S(M,Q, k) \leq \sum_{d \leq Q/2} \dfrac{1}{d} \sum_{\substack{2 \leq q \leq Q/d \\ (q,b^3-b)=1}} \dfrac{1}{q} \sum_{h(q)}^*\Phi_M\left(\dfrac{h}{q} + \dfrac{k}{b^3-b}\right). 
	$$
	We now split the sum over $2 \leq q \leq Q/d$ into $\ll \log Q/d$ sums over dyadic segments $(R/2, R]$ with $2 \leq R \leq Q/d$. Let $0 < \epsilon \leq 1/15$ and consider the contribution of those $q$ in any one such interval $(R/2, R]$.
	Bounding $1/q$ and applying Proposition \ref{prop: average bound} with $\delta = 2/5 - \epsilon$, we get
	\begin{align*}
		&\sum_{\substack{R/2 < q \leq R \\ (q,b^3-b)=1} }\dfrac{1}{q} \sum_{h(q)}^*\Phi_M\left(\dfrac{h}{q} + \dfrac{k}{b^3-b}\right) \leq \dfrac{2}{R} \sum_{\substack{2 \leq q \leq R \\ (q,b^3-b)=1}} \sum_{h(q)}^*\Phi_M\left(\dfrac{h}{q} + \dfrac{k}{b^3-b}\right)\\
		&\ll_{b,\epsilon} 
		\left(Rb^{(\frac{3}{5} + \epsilon - \sigma_1(b,\epsilon))M} + \dfrac{b^M}{R^{\sigma_1(b,\epsilon)}}\right)\exp\left(-\sigma_\infty(b)\dfrac{M}{\log R}\right)
		\\
		&\leq
		Rb^{(\frac{3}{5} + \epsilon)M} + b^M \exp\left(-c(b,\epsilon)\sqrt{M}\right)
	\end{align*}
	for some $\sigma_1(b,\epsilon),  c(b,\epsilon), \sigma_\infty(b) > 0$ depending at most on $b,\epsilon$. Note we used $e^{-\sigma_1(b,\epsilon)\log R - \sigma_\infty(b)M/\log R} \leq e^{-c(b,\epsilon)\sqrt{M}}$ for $R \geq 2$ and $M\geq 0$. Then the combined contribution of the $\ll \log (Q/d)$ dyadic pieces to the sum over $2 \leq q \leq Q/d$ is 
	$$
	\ll_{b,\epsilon} \dfrac{Q}{d}b^{(\frac{3}{5} + \epsilon )M} + b^M \exp\left(-c(b,\epsilon)\sqrt{M}\right) \log Q.
	$$
	Dividing this by $d$ and summing over $1 \leq d \leq Q/2$ we obtain
	$$
	S(M,Q, k) \ll_{b,\epsilon} Qb^{(\frac{3}{5} + \epsilon )M} + b^M \exp\left(-c(b,\epsilon)\sqrt{M}\right) \log^2 Q
	$$ 
	uniformly in $k \in \mathbb{Z}$. Finally summing over $M,N$ on the right hand side of (\ref{eqn: average discrepancy}) yields, for any $1 \leq Q \leq x^{1/5}$ and any $0 < \epsilon \leq 1/15$, 
	$$
	\sum_{\substack{q \leq Q \\ (q,b^3-b)=1}} \sup_{ y \leq x} \max_{a \in \mathbb{Z}} \left|\#\mathscr{P}_b^*(y,a,q) - \dfrac{\#\mathscr{P}_b^*(y)}{q}\right| \ll_{b,\epsilon} Qx^{\frac{3}{10} + \frac{\epsilon}{2}} + \sqrt{x} e^{-k_{b,\epsilon} \sqrt{\log x}},
	$$
	where $k_{b,\epsilon} > 0$ is some value depending only on $b,\epsilon$. Since the above holds for $0 < \epsilon \leq 1/15$, it also holds for $\epsilon > 1/15$. Now the result follows when we let $Q = x^{1/5 - \epsilon}$ and use the fact $\sqrt{x} \ll_b \#\mathscr{P}_b^*(x)$ from Lemma \ref{lem: number of coprimes}. 
\end{proof}	 

\section{Equidistribution with square moduli}\label{sec: square}

Before we can prove Theorem \ref{thm: at most 6 primes} we also need Proposition \ref{prop: square equidistribution} below.
Its proof is a direct application of the estimate of Baier-Zhao \cite{Baier-Zhao} for the large sieve with square moduli (see Lemma \ref{lem: Baier-Zhao} below) and the $L^\infty$-type bound of Col \cite{Col} in Proposition \ref{prop: Linfty bound}. Most definitely, it is possible to replace the $\log^A x$ below with something such as $\exp(k_{b,\epsilon} \sqrt{\log x})$ using Col's bound in Proposition \ref{prop: Linfty bound} and our argument in the proof. Nevertheless this is more than sufficient for our purposes. 

\begin{prop}\label{prop: square equidistribution}
	For any $x \geq 2$ and $\epsilon_0 > 0$,
	$$
	\sum_{\substack{q \leq x^{1/4-\epsilon_0} \\ (q,b^3-b)=1}}\mu^2(q) \sup_{y \leq x}\max_{a \in \mathbb{Z}} \left| \sum_{n\in \mathscr{P}_b^*(y)} \left(\mathbf{1}_{n \equiv a (q^2)} - \dfrac{1}{q^2}\right)\right| \ll_{A,b,\epsilon_0} \dfrac{\#\mathscr{P}_b^*(x)}{\log^A x}
	$$
	for any $A > 0$.
\end{prop}

We need the following two lemmas. The first is due to Baier-Zhao \cite{Baier-Zhao} and Lemma \ref{lem: average of Phi with square moduli} below is a direct application of it. 

\begin{lem}[Baier-Zhao \citelist{\cite{Baier-Zhao}, Theorem 1}]\label{lem: Baier-Zhao}
For any integers $M,Q,N$ with $Q,N \geq 1$ and any sequence $(\gamma_n)$ of complex numbers,
$$
\sum_{q \leq Q}\sum_{a(q^2)}^* \left|\sum_{M < n \leq M+N} \gamma_n e_{q^2}(a n)\right|^2 \ll_\epsilon \mathscr{K}_\epsilon(Q,N) \sum_{M <n \leq M+N}|\gamma_n|^2
$$
for any $\epsilon > 0$, where
$$
\mathscr{K}_\epsilon(Q,N) = (QN)^\epsilon \left(Q^3 + N + \min\left\{N\sqrt{Q}, \sqrt{N}Q^2\right\}\right).
$$	
\end{lem}

\begin{lem}\label{lem: average of Phi with square moduli}
For any $Q \geq 1$ and integer $N \geq 0$,
\begin{equation}\label{eqn: L2 bound with square moduli}
\sup_{\beta \in \mathbb{R}} \sum_{\substack{q \leq Q \\ (q,b)=1}} \sum_{a(q^2)}^* \Phi_N \left(\dfrac{a}{q^2} + \beta\right) \ll_{b,\epsilon} Q^{3 + \epsilon}b^{N/2} + Q^{15/8 + \epsilon} b^N
\end{equation}
for any $\epsilon > 0$.	
\end{lem}

\begin{proof}
We may assume that $N$ is arbitrarily large and set $M = \lfloor 4N/9 \rfloor > 1$.
Let us first consider the case when $Q \gg_b b^M$. Here we split the product as $\Phi_N = P_1 P_2$, where $P_1$ is the product over $1 \leq n \leq M$ and $P_2$ is the product over $M < n < N$. By the Cauchy-Schwarz inequality, the left hand side above is
$$
\leq \left(\sum_{q\leq Q} \sum_{a(q^2)}P_1^2\left(\dfrac{a}{q^2} + \beta\right)\right)^{1/2}\left(\sum_{\substack{q \leq Q \\ (q,b)=1}}\sum_{a (q^2)}^* P_2^2\left(\dfrac{a}{q^2} + \beta\right)\right)^{1/2}.
$$
Arguing similarly as done in the proof of Lemma \ref{lem: L2 bound} and using the assumption $Q \gg_b b^M$, one can show that
$$
\sum_{q\leq Q} \sum_{a(q^2)}P_1^2\left(\dfrac{a}{q^2} + \beta\right) \ll_\epsilon Q^{3+\epsilon}b^M
$$
for any $\epsilon > 0$. With regards to the second sum, as we have seen before,
$$
\sum_{\substack{q \leq Q \\ (q,b)=1}}\sum_{a (q^2)}^* P_2^2\left(\dfrac{a}{q^2} + \beta\right) = \sum_{\substack{q \leq Q \\ (q,b)=1}}\sum_{a (q^2)}^* \Phi_{N-M}^2\left(\dfrac{a}{q^2} + \beta b^M\right).
$$
We may write 
$$
\Phi_{N-M}\left(\dfrac{a}{q^2} + \beta b^M\right) = \left|\sum_{0 \leq n \leq b^{2(N-M)}}\gamma_n e_{q^2}(an)\right|,
$$
where the coefficients $\gamma_n$ are given by 
$$
\gamma_n = e\left(\beta b^M n\right)\mathbf{1}_{n \in \mathcal{S}}
$$
with
$$
\mathcal{S} = \left\{ \sum_{1 \leq m < N-M} c_m \left(b^m + b^{2(N-M) -m}\right) \ : \ 0 \leq c_1, \ldots, c_{N-M-1} < b\right\}.
$$
Then it follows from Lemma \ref{lem: Baier-Zhao} that
$$
\sum_{\substack{q \leq Q \\ (q,b)=1}}\sum_{a (q^2)}^* P_2^2\left(\dfrac{a}{q^2} + \beta\right) \ll_{\epsilon}\mathscr{K}_\epsilon\left(Q, b^{2(N-M)}\right) b^{N-M}.
$$
From the definition of $\mathscr{K}_\epsilon$ in Lemma \ref{lem: Baier-Zhao} and the assumption $Q \gg_b b^M$ with $M = \lfloor 4N/9 \rfloor$, one can show that
$$
\mathscr{K}_\epsilon\left(Q, b^{2(N-M)}\right) \ll_\epsilon Q^{\epsilon} \left(Q^3 + b^{2(N-M)}\sqrt{Q}\right) \ll_b Q^{3 + \epsilon}. 
$$
Now (\ref{eqn: L2 bound with square moduli}) follows after multiplying the two bounds involving $P_1,P_2$ and taking square roots. 

Consider now the case when $Q \ll_b b^M$. Here we may find an integer $L > 1$ such that $Q \asymp_b b^{4(N-L)/9}$. As done in a previous section, we split the product $\Phi_N$ according to whether $n \leq L$ or $L < n < N$. Bounding the product over $1 \leq n \leq L$ trivially by $b^L$, noting that 
\begin{align*}
\mathscr{Z} &:= \sum_{\substack{q \leq Q \\ (q,b)=1}}\sum_{a(q^2)}^*\prod_{L < n < N} \phi_b\left(\left(\dfrac{a}{q^2} + \beta\right)\left(b^n + b^{2N-n}\right)\right)\\ &= \sum_{\substack{q \leq Q \\ (q,b)=1}}\sum_{a(q^2)}^* \Phi_{N-L} \left(\dfrac{a}{q^2} + b^L \beta\right)
\end{align*}
and using (\ref{eqn: L2 bound with square moduli}) on $\mathscr{Z}$ with $N$ there substituted with $N-L$ (we may do this since now $Q \asymp_b b^{4(N-L)/9}$ by our assumption) we obtain
$$
\sum_{\substack{q \leq Q \\ (q,b)=1}} \sum_{a(q^2)}^* \Phi_N \left(\dfrac{a}{q^2} + \beta\right) \ll_{b,\epsilon} Q^{3+\epsilon} b^{\frac{N+L}{2}} + Q^{15/8 + \epsilon} b^N.
$$
The assumption $Q \asymp_b b^{4(N-L)/9}$ implies $b^{L/2} \asymp_b b^{N/2}/Q^{9/8}$, whence the left-most term, of the right hand side above, is $\ll_b Q^{15/8+\epsilon} b^N$. It follows that (\ref{eqn: L2 bound with square moduli}) also holds when $Q \ll_b b^M$.
\end{proof}	

\begin{proof}[{\bf Proof of Proposition \ref{prop: square equidistribution}}]
	For the sake of brevity let $U = x^{1/4-\epsilon_0}$ and let $E(U)$ denote the left hand side of the inequality in Proposition \ref{prop: square equidistribution}. Now
	the Fourier expansion 
	$$
	\mathbf{1}_{n \equiv a (q^2)} = \dfrac{1}{q^2} \sum_{h (q^2)}e\left(\dfrac{-ah}{q^2}\right) e\left(\dfrac{nh}{q^2}\right)
	$$
	gives
	$$
	E(U) \leq \sum_{\substack{q \leq U \\ (q,b^3-b)=1}} \dfrac{\mu^2(q)}{q^2} \sum_{1 \leq h < q^2} \sup_{y\leq x}\left|\sum_{n \in \mathscr{P}_b^*(y)}e\left(\dfrac{hn}{q^2}\right)\right|.
	$$
	Splitting the sum according to the GCD of $h,q$ and substituting variables gives
	$$
	E(U) \leq \sum_{\substack{qr \leq U \\ (qr,b^3-b)=1}} \dfrac{\mu^2(qr)}{q^2r^2} \sum_{\substack{1 \leq h < q^2 r \\ (h,q)=1}} \sup_{y\leq x} \left|\sum_{n \in \mathscr{P}_b^*(y)}e\left(\dfrac{hn}{q^2 r}\right)\right|.
	$$
	We now split the last sum according to the GCD of $h,r$ and substitute variables as before, obtaining
	$$
	E(U) \leq \sum_{\substack{dqr \leq U \\ qr > 1\\ (dqr,b^3-b)=1}} \dfrac{\mu^2(dqr)}{d^2q^2r^2} \sum_{\substack{h( q^2 r)}}^* \sup_{y\leq x}\left|\sum_{n \in \mathscr{P}_b^*(y)}e\left(\dfrac{hn}{q^2 r}\right)\right|.
	$$
	By the fact
	$$\mathscr{P}_b^*(y) = \{n \in \mathscr{P}_b^0(y) \ : \ (n,b^3-b)=1\},$$ 
	the M\"{o}bius inversion formula $\mathbf{1}_{(n,b^3-b)=1} = \sum_{s \mid (n,b^3-b)}\mu(s)$, the Fourier expansion $\mathbf{1}_{s \mid n} = \frac{1}{s} \sum_{k (s)} e_s(nk)$ and a trivial bound for the $d$-sum, it follows
	$$
	E(U) \ll \tau(b^3-b)\max_{ k \in \mathbb{Z}} \sum_{\substack{1 < qr \leq U \\ (qr,b^3-b)=(q,r)=1}} \dfrac{1}{q^2r^2} \sum_{\substack{h( q^2 r)}}^*\sup_{y\leq x} \left|\sum_{n \in \mathscr{P}_b^0(y)}e\left(\dfrac{hn}{q^2 r} + \dfrac{k n}{b^3-b}\right) \right|.
	$$
	By Lemma \ref{lem: bound for exp sum of pals}, the Chinese remainder theorem and dyadic decompositions of the intervals, we then have
	$$
	E(U) \ll_b (\log x)^4 \sup_{\substack{QR \ll U \\ 0 \leq N \leq \frac{1}{2}\log_b x \\ k \in \mathbb{Z}}} \mathscr{T}(N, Q,R,k) ,
	$$
	where
	$$
	\mathscr{T}(N, Q,R,k) = \dfrac{1}{Q^2 R^2}\sum_{\substack{ q \leq Q \\ r \leq R \\ qr > 1 \\ (qr,b^3-b)=(q,r)=1}} \sum_{g (q^2)}^* \sum_{h (r)}^*\Phi_N\left(\dfrac{g}{q^2} +  \dfrac{h}{r} + \dfrac{k}{b^3-b}\right).
	$$
	Let $B > 1000$ and $ C > 1000B$. We bound $\mathscr{T}(N, Q,R,k)$ according to the following three cases of $Q,R$.  
	
	{\bf Case 1} ($Q \ll \log^B x$ and $R \ll \log^C x$): Here Proposition \ref{prop: Linfty bound} implies 
	$$\mathscr{T}(N, Q,R,k) \ll_{b,C,D} \sqrt{x} \log^{-D}x$$ for any $D > 0$. 
	
	{\bf Case 2} ($Q \ll \log^B x$ and $R \gg \log^C x$): Taking the worst possible value of $g/q^2$ and applying Proposition \ref{prop: average bound} with $\delta = \frac{2}{5} -\epsilon$ there, one has
	\begin{align*}
	\mathscr{T}(N, Q,R,k) &\leq \dfrac{Q}{R^2} \sup_{\beta \in \mathbb{R}} \sum_{\substack{r \leq R \\ (r,b)=1}} \sum_{h (r)}^*\Phi_N\left( \dfrac{h}{r} + \beta\right) \ll_{b,\epsilon} Q x^{\frac{3}{10} + \epsilon} + \dfrac{Q\sqrt{x}}{R}\\
	&\ll_C \dfrac{\sqrt{x}}{\log^{C-B} x} 
	\end{align*}
	for $\epsilon > 0$ small and fixed.
	
	\begin{rmk}
	We employed Proposition \ref{prop: average bound} for the sake of convenience, but it is unnecessary here. It suffices instead to proceed via a Cauchy-Schwarz argument combined with Lemma \ref{lem: L2 bound}.
	\end{rmk} 

{\bf Case 3} ($Q \gg \log^B x$): Note
\begin{equation}\label{eqn: case 3}
\mathscr{T}(N, Q,R,k) \leq  \dfrac{1}{Q^2} \sup_{\beta \in \mathbb{R}} \sum_{\substack{q \leq Q \\ (q,b)=1}} \sum_{g (q^2)}^* \Phi_N\left(\dfrac{g}{q^2} + \beta\right)
\end{equation} 
while Lemma \ref{lem: average of Phi with square moduli} and the assumptions $2N \leq \log_b x$ and $\log^B x \ll Q \ll U = x^{1/4-\epsilon_0}$ imply this is
$$
\ll_{b,\epsilon} Q^{1+\epsilon} x^{1/4} + \dfrac{\sqrt{x}}{Q^{\frac{1}{8}-\epsilon}} \ll_{ B, \epsilon_0} \sqrt{x}\log^{-\frac{B}{8} + 1} x
$$ 
after a suitable choice of $\epsilon$ in terms of $B$ and $\epsilon_0$. Now Proposition \ref{prop: square equidistribution} follows after choosing $B,C,D$ sufficiently large in terms of $A$. 	
\end{proof}

\section{Palindromic almost-primes}\label{sec: final proof}

We conclude the work with a proof of Theorem \ref{thm: at most 6 primes}.
To this end we use the following version of the almost-prime sieve. 

\begin{lem}[Almost-prime linear sieve with Richert's weights]\label{lem: sieve lemma}
	Let $\mathcal{A}$ be a finite sequence of integers in $[1,x]$.
	For each integer $d \geq 1$, set
	$$
	A_d := \#\left\{n \in \mathcal{A} \ : \ d \mid n\right\}.
	$$
	Suppose for each such $d$ that 
	$$
	A_d = Xg(d) + r_d
	$$
	for some $r_d$, real number $X$ and some multiplicative function $g$ satisfying $0 \leq g(p) < 1$ for each prime $p$ and
	$$
	\prod_{u \leq p < v} (1 - g(p))^{-1} \leq K\dfrac{\log v}{\log u}
	$$
	for any $2 \leq u < v \leq x$, where $K > 1$ is some constant. 
	
	For an integer $r \geq 2$, set 
	$$
	u_r := 1 + 3^{-r}
	$$
	and
	$$
	\Delta_r := r + \log_3\left(\dfrac{3u_r}{4}\right).
	$$
	
	Suppose that 
	$$
	\sum_{d \leq D} |r_d| \ll \dfrac{X}{\log^3 x}
	$$
	and 
	$$
	\sum_{D^{1/4}  \leq p \leq D^{1/u_r}} A_{p^2} \ll \dfrac{X}{\log^3 x}
	$$ 
	for some $D$ satisfying  $D \geq x^{1/\Delta_r + \epsilon}$ for some $\epsilon > 0$. Then
	$$
	\sum_{\substack{n \in \mathcal{A}\\ P^-(n) \geq D^{1/4} \\ \Omega(n) \leq r}}1 \asymp_{r,\epsilon} X \prod_{p < x}(1 - g(p)).
	$$
\end{lem}

\begin{proof}
	For the lower bound see for instance Corollaries 1.1 and 1.2 in Section 5 of Greaves \cite{Greaves}. For the upper bound see Corollary 1.1 in Section 1 of \cite{Greaves}.
\end{proof}	

\begin{proof}[{\bf Proof of Theorem \ref{thm: at most 6 primes}}]
	We may assume that $x \gg_b 1$ is large enough so that $z := x^{1/21}$ is much larger than $b^3-b$. 
	In this case, 
	$$
	S := \sum_{\substack{n \in \mathscr{P}_b(x) \\ P^-(n) \geq z\\ \Omega(n)\leq 6}} 1=
	\sum_{\substack{n \in \mathscr{P}_b^*(x) \\ P^-(n) \geq z \\ \Omega(n)\leq 6}} 1.
	$$	
	For any integer $d \geq 1$, set
	$$
	A_d := \sum_{\substack{n \in \mathscr{P}_b^*(x) \\ d \mid n }} 1 = \#\mathscr{P}_b^*(x,0,d).  
	$$ 
	We can write this as 
	$$
	A_d = g(d)\#\mathscr{P}_b^*(x) +r_d, 
	$$
	where
	$$
	g(d) = \dfrac{\mathbf{1}_{(d,b^3-b)=1}}{d}
	$$
	and
	$$
	r_d =  \#\mathscr{P}_b^*(x,0,d) - g(d) \#\mathscr{P}_b^*(x).
	$$
	Clearly $g$ is multiplicative, satisfies $0 \leq g(p) < 1$ for each prime $p$ and 
	$$
	\prod_{u \leq p < v} (1 - g(p))^{-1} \leq \prod_{u \leq p < v} \left(1 - \dfrac{1}{p}\right)^{-1} = \dfrac{\log v}{\log u}\left(1 + O\left(\dfrac{1}{\log u}\right)\right)
	$$
	for any $2 \leq u < v \leq x$,
	by the Mertens' theorems. 
	
	Set $D := z^4 = x^{4/21} = x^{1/5 - 1/105}$. Since every number in $\mathscr{P}_b^*$ is coprime to $b^3-b$, 
	$$
	\sum_{d \leq D}|r_d| = \sum_{\substack{d \leq D \\ (d,b^3-b)=1}} \left|\#\mathscr{P}_b^*(x,0,d) - \dfrac{ \#\mathscr{P}_b^*(x)}{d}\right| \ll_{b}  \dfrac{\#\mathscr{P}_b^*(x)}{\log^{10}x} 
	$$ 
	by Theorem \ref{thm: equidistribution}. It also follows from Proposition \ref{prop: square equidistribution} that
	$$
	\sum_{D^{1/4} \leq p \leq D} A_{p^2} \ll_b \dfrac{\#\mathscr{P}_b^*(x)}{\log^{10} x}.
	$$
	For $\Delta_6$ defined as in Lemma \ref{lem: sieve lemma}, one can check that $x^{1/\Delta_6 + 1/100} \leq D$. Then all the assumptions of Lemma \ref{lem: sieve lemma} are satisfied. It gives
	$$
	S \asymp_{b} \#\mathscr{P}_b^*(x) \prod_{p < x} \left(1 - g(p)\right) \asymp_b \dfrac{ \#\mathscr{P}_b(x)}{\log x}
	$$
	as required. 
\end{proof}

%%%%%%%%%%%%%%%%%%%%%%%%%%%%%%%%%%%%%%%%%%%%%%%%%%%%%%%%%%%%%%%%%%%%%%%%%%%%%%%%%%%%%%%%%%%%%%%%%%%%%%%%%%%%%%%%%%%%%%%%%%%%%%%%%%%%%%%%%%%%%%%%%%

\end{document}